\documentclass[reqno,11pt]{amsart}
\usepackage{amsmath,amssymb,amsthm, verbatim}
\usepackage{url}
\usepackage[usenames, dvipsnames]{color}

\usepackage[letterpaper,hmargin=1in,vmargin=1in]{geometry}

\usepackage{graphicx}
\usepackage{enumerate,pinlabel}
\usepackage{mathrsfs,graphicx,url}
\usepackage[usenames,dvipsnames]{xcolor}
\usepackage[colorlinks=true,linkcolor=Black,citecolor=Black, urlcolor=Black]{hyperref}

\numberwithin{equation}{section}

\theoremstyle{plain}
\newtheorem{theorem}{Theorem}[section]
\newtheorem*{theorem*}{Theorem}
\newtheorem*{proposition*}{Proposition}
\newtheorem*{lemma*}{Lemma}
\newtheorem{lemma}[theorem]{Lemma}
\newtheorem{proposition}[theorem]{Proposition}

\newtheorem{corollary}[theorem]{Corollary}
\newtheorem*{corollary*}{Corollary}

\theoremstyle{definition}

\theoremstyle{remark}
\newtheorem{remark}[theorem]{Remark}

\newcommand{\e}{\epsilon}

\newcommand{\ep}{\varepsilon}
\newcommand{\NN}{\mathbb{N}}
\newcommand{\R}{\mathbb{R}}
\newcommand{\US}{\mathbb{S}}
\newcommand{\C}{\mathbb{C}}

\newcommand\supp{\mathop{\rm supp}}

\newcommand\real{\mathop{\rm Re}}
\newcommand\imag{\mathop{\rm Im}}
\newcommand\sgn{\mathop{\rm sgn}}
\newcommand*{\defeq}{\mathrel{\vcenter{\baselineskip0.5ex \lineskiplimit0pt

                     \hbox{\scriptsize.}\hbox{\scriptsize.}}}%
                     =}

\newcommand*{\qefed}{=\mathrel{\vcenter{\baselineskip0.5ex \lineskiplimit0pt

                     \hbox{\scriptsize.}\hbox{\scriptsize.}}}}

\title[Logarithmic wave decay]{Logarithmic wave decay for short range wavespeed perturbations with radial regularity}

\author{Gayana Jayasinghe}
\thanks{No affiliation, Kandy, Sri Lanka. \texttt{mgsjayasinghe@gmail.com}}

\author{Katrina Morgan}
\thanks{Department of Mathematics, Temple University, Philadelphia, PA, USA. \texttt{morgank@temple.edu}}

\author{Jacob Shapiro}
\thanks{Department of Mathematics, University of Dayton, Dayton, OH, USA. \texttt{jshapiro1@udayton.edu}. \emph{Corresponding author.}}

\author{Mengxuan Yang}
\thanks{Department of Operations Research \& Financial Engineering, Princeton University, Princeton, NJ, USA. \texttt{yangmx@princeton.edu}}

\begin{document}

\maketitle





\begin{abstract}
We establish logarithmic local energy decay for wave equations with a varying  wavespeed in dimensions two and higher, where the wavespeed is assumed to be a short range perturbation of unity with mild radial regularity. The key ingredient is H\"older continuity of the weighted resolvent for real frequencies $\lambda$, modulo a logarithmic remainder in dimension two as $\lambda \to 0$. Our approach relies on a study of the resolvent in two distinct frequency regimes. In the low frequency regime, we derive an expansion for the resolvent using a Neumann series and properties of the free resolvent. For frequencies away from zero, we establish a uniform resolvent estimate by way of a Carleman estimate.

\medskip

\noindent \textbf{Acknowledgements:} We thank Kiril Datchev for helpful discussions.

\medskip

\end{abstract}

\section{Introduction} \label{introduction section}

The goal of this article is to establish sharp local energy decay for the solution to the variable coefficient wave equation,
\begin{equation} \label{wave eqn}
\begin{cases}
 (\partial_t^2  - c^2(x)\Delta + V(x))u(x,t) = 0, \qquad (x,t) \in \R^n \times \R, \\
u(x,0) = u_0(x), \, \partial_t u(x,0) = u_1(x).
\end{cases} 
\end{equation}
in dimension $n  \ge 2$, where $\Delta \le 0$ is the Laplacian on $\R^n$.

We impose the following regularity and decay on the \textit{wavespeed} $c(x)$:
\begin{equation} \label{c bd above and below}
c, \, c^{-1} \in L^\infty(\R^n; (0, \infty)),
\end{equation}
and 
\begin{equation} \label{one minus c short range} 
|1 - c(x)| \le C \langle x \rangle^{-\delta_0} 
\end{equation}
for some $C> 0$ and $\delta_0 > 0$, where $\langle x \rangle \defeq (1 + |x|^2)^{1/2}$. Furthermore, the radial derivative $\partial_r c$ (where $r = |x|$) defined in the sense of distributions, should belong to $L^\infty(\R^n)$ and satisfy
\begin{equation} \label{partialr c short range}
|\partial_r  c(x)|  \le C\langle x \rangle^{-\delta_1}. 
\end{equation}
for some $C> 0$ and $\delta_1 > 0$. The potential $V(x)$ is assumed to be nonnegative. In dimension $n=2$, we require $V \equiv 0$, while for $n \ge 3$, we assume $V$ has sufficient decay at infinity, as specified in Theorem \ref{LED thm}.

\begin{theorem} \label{brief decay}
Let $s > 0$. Assume the wavespeed $c$ meets conditions \eqref{c bd above and below}, \eqref{one minus c short range} with $\delta_0 > 2$, and \eqref{partialr c short range} with $\delta_1 > 1$. Let the potential $V$ be as specified in Theorem \ref{LED thm} below. Then there exsits $C > 0$, such that for any initial data $(u_0, u_1)$ with $\langle x \rangle^s u_0 \in H^2(\R^n)$ and $ \langle x \rangle^s u_1 \in H^1(\R^n)$, where $H^2$ and $H^1$ are the standard Sobolev spaces, the corresponding solution $u$ to \eqref{wave eqn} obeys
\begin{equation} \label{prelim decay}
\begin{split}
\| \nabla \langle x \rangle^{-s} u(\cdot,t)\|_{L^2(\R^n)}  &+  \| \langle x \rangle^{-s} \partial_t u(\cdot,t)\|_{L^2(\R^n)}\\
&\le \frac{C}{1 + \log \langle t \rangle}\big( \|  \langle x \rangle^s u_0 \|_{H^2(\mathbb{R}^n)} + \|  \langle x \rangle^s u_1\|_{H^1(\mathbb{R}^n)} \big).
\end{split}
\end{equation}
\end{theorem}


Previously, the third author established \eqref{prelim decay}, with a compactly supported weight, when $\nabla c \in L^\infty(\R^n)$, $n \ge 2$, and $c = 1$ outside a compact set \cite[Theorem 1]{sh18}. Thus the novelty of Theorem \ref{brief decay} is that it extends this result to a more general spatial weight while relaxing the conditions on the wavespeed.

Theorem \ref{brief decay} is a consequence of Theorem \ref{LED thm} in Section \ref{stmt of main thm subsection}, and we provide the proof of this implication in Section \ref{proof of prelim decay appendix}. In addition, we show that \eqref{prelim decay} can be strengthened: under additional regularity assumptions on the initial data with respect to $-c^2 \Delta + V$, one obtains more decay in time. More precisely, the power of the inverse logarithmic term on the right-hand side of \eqref{prelim decay} can be increased, at the cost of replacing the norm on the initial data by one involving higher powers of $-c^2 \Delta + V$. Furthermore, if we assume $s>1$, the gradient term in the left-hand side of \eqref{prelim decay} may be replaced by $\|\langle x\rangle^{-s}u(\cdot,t)\|_{H^{1}}$. 



In addition, the Carleman estimate developed in Section \ref{Carleman estimate section}, and thus Theorems \ref{brief decay} and \ref{LED thm}, remain valid under a regularity condition on $c(x)$ slightly weaker than \eqref{partialr c short range}. Specifically, for each direction $\theta \in \US^{n-1}$, the profile $r \mapsto c(r\theta)$,  may have jump discontinuities. These are permissible provided they occur within a fixed compact set of radii and that the total radial variation is controlled uniformly across all directions. See \eqref{V prime long range} for the precise assumption.

Logarithmic decay was first obtained by Burq for smooth, compactly supported metric perturbations of the Laplacian in dimensions $n \ge 2$ \cite{bu98}, and later extended to long range metrics analytic at infinity, provided the initial data is localized away from zero frequency \cite{bu02}. Both cases allow for a smooth, compact, Dirichlet obstacle. Cardoso and Vodev expanded the result of \cite{bu02} to manifolds, and without the analyticity assumption \cite{cavo04}. Bouclet showed that for smooth long range metrics on $\R^n$, $n \ge 3$, the spectral localizer is not necessary \cite{bo11}. Recently, Christiansen, Datchev, Morales, and the last author generalized the method in \cite{cdy} and revisited Burq's original setting of compactly supported perturbations, establishing logarithmic decay in dimension two without any regularity assumption at zero frequency \cite{cdmy}.

If $n \ge 2$ and there is no condition on the radial derivative of the wavespeed, then only slower local energy decay rates are known. Such results require \eqref{c bd above and below} and in addition $c \equiv 1$ outside of a compact set. The sharpest decay rate known in that case is $(\log( \log t)/\log t)^{3/4}$, $t \gg 1$; it improves to $(\log( \log t)/\log t)^{(\alpha + 3)/4}$ if $c$ is H\"older continuous with H\"older exponent $0 < \alpha < 1$ \cite[Corollary 1.5]{vo20}.  On the other hand, if we suppose \eqref{c bd above and below} and $c = 1$ outside a compact set, and $c$ is radially symmetric, it follows from the resolvent estimates in \cite{vo22, dgs23} that the local energy decays like $1/\log t$. These decay rates contrast with the case $n=1$, where exponential decay occurs if the wavespeed has bounded variation and equals one outside a compact set \cite{dash23}.

The proof of Theorem \ref{LED thm} shows that if we localize $u(\cdot, t)$ away from zero frequency, we obtain logarithmic decay in any dimension $n \ge 2$, provided $\limsup_{|x| \to \infty}|1 - c(x)| = 0$ as well as $\delta_1 > 1$ in \eqref{partialr c short range}. Our requirement that $1 - c = O(\langle x \rangle^{-\delta_0})$ for some $\delta_0 > 2$ arises from our treatment of the low-frequency regime. Specifically, we use a Neumann series to relate the resolvent of $-c^2 \Delta + V$ to that of $-\Delta + V$ (Section \ref{resolvent behavior section}). Under our short range assumptions on $V$, the low-frequency behavior of this latter resolvent can be understood from the asymptotics of the free resolvent (Appendices \ref{free resolv low freq appendix} and \ref{perturb resolv low freq appendix}). On the other hand, for $n \ge 3$, Bony and H\"afner used the Mourre method to establish a low frequency resolvent bound for $-c(x)\sum_{i,j =1}^n \partial_{x_i} g_{ij}(x) \partial_{x_j}$, provided $c$ and the $g_{ij}$ are smooth with $|\partial^\alpha_x(1 - c)| +  \sum_{i,j =1}^n|\partial^\alpha_x g_{i,j}|  = O(\langle x \rangle^{-\delta - |\alpha|})$ for some $\delta> 0$ and all multi-indices $\alpha$ \cite{boha10}.

Logarithmic decay arises in a variety of contexts, including transmission problems \cite{be03}, damped waves \cite{bujo16, wa24}, and general relativity \cite{hosm13, mo16, ga19}. Its significance lies in the fact that it is often the optimal decay rate, particularly in settings where no nontrapping assumption is imposed on the dynamics generated by the Hamiltonian of $-c^2\Delta$. In our case, however, the Hamiltonian flow may not be well-defined, since $c$ may lack the regularity required for classical existence and uniqueness. More broadly, the saturation of logarithmic decay is closely tied to the presence of resonances exponentially close to the real axis. This connection was first identified by Ralston in the case of radial wavespeeds \cite{ra71}, with related constructions developed in \cite{hosm14, ke16, ke20, ben21, dmms21, kugu21}.

We briefly outline key developments in the study of local energy decay for solutions to the wave equation. Foundational results are due to Morawetz and her work with Lax and Philips \cite{mor61, mor62, lmp63}, establishing decay of waves exterior to nontrapping obstacles. Beginning with Keel, Smith, and Sogge \cite{kss02}, local energy decay became a standard tool in the analysis of nonlinear wave equations. Local energy decay is deeply connected with resolvent behavior \cite{mst20,lsv25} and has been used to prove Strichartz estimates (e.g. Marzuola-Metcalfe-Tataru-Tohaneanu \cite{mmtt}) as well as pointwise decay estimates (e.g. Tataru). We conclude by pointing to a broader body of influential work and surveys that chart the development of wave decay theory: \cite[Epilogue]{lp89}, \cite[Chapter X]{va}, \cite{dr}, \cite{tat}, \cite{hizw17}, \cite{dz}, \cite{vasy}, \cite{sch21}, \cite{klainerman}, \cite{hin}, \cite{lo24}.

\subsection{Statement of main theorem and strategy of proof} \label{stmt of main thm subsection}

Our proof of weighted energy decay proceeds via resolvent estimates and spectral methods. The spatial component $-c^2 \Delta$ of the wave operator is formally symmetric on the weighted space $L^2_c(\R^n) \defeq L^2(\R^n; c^{-2}(x)dx)$, and is self-adjoint and nonnegative when equipped with domain the Sobolev space $H^2(\R^n)$ \cite[Proposition A.1]{sh18}. By the Kato-Rellich Theorem, the same remains true if one adds $V \in L^\infty(\R^n ; [0, \infty))$. Note that $L_c^2(\R^n)$ coincides with the standard space $L^2(\R^n) = L^2(\R^n ; dx)$ since both $c$ and $c^{-1}$ are bounded.

Setting $G \defeq -c^2 \Delta + V$, the solution $u(\cdot, t)$ to \eqref{wave eqn} can be expressed via the spectral theorem as 
\begin{equation*}
u(\cdot, t) = \cos(t \sqrt{G}) u_0 + \frac{\sin(t \sqrt{G})}{\sqrt{G}} u_1.
\end{equation*}
To quantify decay, we localize spectrally to a window whose width grows slowly in time. Let $\mathbf{1}_I$ denote the characteristic function of an interval $I \subseteq \R$. Our main technical result is

\begin{theorem} \label{LED thm}
Let $s > 1$ and $m \ge 0$. Assume $c$ satisfies \eqref{c bd above and below}, \eqref{one minus c short range} with $\delta_0 > 2$, and \eqref{partialr c short range} with $\delta_1 > 1$. Let $V \in L^\infty(\R^n ; [0, \infty))$ and suppose further that there exist constants $C > 0$ and $\rho > 0$ such that
\begin{equation*}
|V(x)| \le C \langle x \rangle^{-\rho}
\end{equation*}
 with
\begin{equation} 
\rho > \begin{cases}
7/2 & \text{if } n = 3, \\
5 & \text{if } n = 4, \\
\max(3, n/2) & \text{if } n \ge 5.
\end{cases}
\end{equation}
In the case $n = 2$, we assume $V \equiv 0$. If, $n =4$, additionally assume that $V$ is Lipschitz, in the sense that the distributional derivatives $\partial_{x_j} V$, $1 \le j \le 4$, belong to $L^\infty(\R^4)$. 

There exist $C, \, \nu, \, \gamma > 0$ so that for $|t| \gg 1$ and $A = A(t) = \gamma \log |t| $,
\begin{align}
\|\langle x \rangle^{-s}  \mathbf{1}_{[0, A^2]}(G) G^{m} \cos(t \sqrt{G})  \langle x \rangle^{-s} \|_{L^2(\R^n) \to H^1(\R^n)} &\le C t^{-\nu}, \label{LED cosine propagator low freq} \\
\|\langle x \rangle^{-s} \mathbf{1}_{[0, A^2]}(G) G^{m} \tfrac{\sin(t \sqrt{G})}{\sqrt{G}} \langle x \rangle^{-s} \|_{L^2(\R^n) \to H^1(\R^n)} &\le C t^{-\nu} \label{LED sine propagator low freq}.
\end{align}
\end{theorem}

\begin{remark}
Since $u(-t, \cdot) = \cos(t \sqrt{G})u_0 + (\sin(t \sqrt{G})/\sqrt{G})(-u_1)$, it suffices to establish \eqref{LED cosine propagator low freq} and \eqref{LED sine propagator low freq} for $t \gg 1$.
\end{remark}

To prove Theorem \ref{LED thm}, we establish H\"older regularity of the boundary values on the real axis of the weighted resolvent $\langle x \rangle^{-s} (-c^2 \Delta + V - \lambda^2)^{-1} \langle x \rangle^{-s}$. Our analysis is split into two frequency regimes.

At low frequency, we construct the required resolvent expansion in three steps. First, the weight condition $s>1$ provides the necessary H\"older continuity for the free resolvent, modulo a logarithmic term in dimension two (Appendix \ref{free resolv low freq appendix}). When $n \ge 3$, the short-range assumptions on $V$ then allow one to transfer this property to the resolvent for $-\Delta+V$ (Appendix \ref{perturb resolv low freq appendix}). Finally, the condition $\delta_0 > 2$ ensures that for small $\lambda$, a Neumann series converges, which relates the resolvent of $-\Delta+V$ to that of the full operator $-c^2\Delta+V$ (Section \ref{resolvent behavior section}).

For frequencies away from zero, we adopt a semiclassical perspective. A formal calculation, treating $\lambda$ as real and letting $h = |\lambda|^{-1}$, motivates relating the original resolvent to a semiclassical one:
\begin{equation*}
\begin{gathered}
(-c^2\Delta + V - \lambda^2)^{-1} = h^{-2} (-h^2 \Delta + V_c + h^2V - 1)^{-1},\\
h \defeq |\lambda|^{-1}, \qquad V_c \defeq 1 - c^{-2}, 
\end{gathered}
\end{equation*}
The necessary H\"older regularity of the associated resolvent is established in Section \ref{Holder cont section} using a Carleman estimate developed in subsections \ref{preliminary calculations subsection} through \ref{Carleman est subsection}.

A feature of the Carleman estimate is its uniformity: it holds for all $h \in (0,h_0]$ with arbitrary $h_0 > 0$, rather than only for sufficiently small $h_0$, as common in the literature (see e.g., \cite{ob24, sh24}). This flexibility stems from an ODE-based construction adapted from \cite[Proposition 3.1]{dadeh16}, which enables control of the second derivative of the Carleman phase. Introduced in \cite{dadeh16} to handle wavespeed discontinuities, this technique plays a central role in our setting. In addition, we use a spatial weight similar to that in \cite{ob24}, which further facilitates explicit computations and contributes to the uniformity of the estimate.

Another aspect of the Carleman phase is that it is constant outside a compact set. It is well known this leads to a exterior weighted estimate for operators such as $-h^2 \Delta + V_c + h^2V - E$, with $E > 0$. If $V_c$ has compact support, $V \equiv 0$, and $n \ge 3$, Remark \ref{optimal ext weight remark} shows our exterior estimate \eqref{ext Carleman est} holds if the weight vanishes on a ball centered at the origin, whose radius grows like $E^{-1/2}$ as $E \to 0$. This scaling is sharp in specific examples \cite{daji20}. The same $E$-dependence was previously obtained for compactly supported potentials that are Lipschitz in the radial variable \cite{gash22b, ob24}. The novelty here is that the same scaling holds for potentials that may be discontinuous along a fixed direction on the sphere, as described above.

The final step of the proof of Theorem \ref{LED thm} is Section \ref{proof LED thm section}, where we combine the H\"older regularity with Stone's formula to obtain \eqref{LED cosine propagator low freq} and \eqref{LED sine propagator low freq}. This approach is due to Cardoso and Vodev \cite[Section 2]{cavo04}.

\subsection{Future directions}
It is natural to ask whether Theorem \ref{brief decay} still holds for smaller values of $\delta_0$ or $\delta_1$, or under weaker regularity assumptions on $c$. Another extension would be to incorporate a potential in dimension two. In our framework, this creates a technical trade-off, requiring a stronger wavespeed decay assumption ($\delta_0 > 4$). A different approach, building on the resolvent expansions in \cite{chda25, jene01}, may be necessary to overcome this.

Separately, one could consider low regularity analogues of the perturbations in \cite{boha10} or the inclusion of an obstacle. Progress on these latter problems would likely require a new type of Carleman estimate, one less reliant on separation of variables. References relevant to these potential developments include \cite{cavo02, rota15, vo20}.

\subsection{List of Notations} 
\label{remark_Notation}
\begin{itemize}
    \item We use $(r, \theta) = (|x|, x/|x|) \in (0, \infty) \times \US^{n-1}$ for polar coordinates on $\R^n \backslash \{0\}$. 
    \item For $u$ defined on a subset of $\R^n$, we write $u(r, \theta) \defeq u(r \theta)$ and $u' \defeq \partial_ru$ for radial derivatives.
    \item $\langle x \rangle \defeq (1 + |x|^2)^{1/2}$.
    \item For $r > 0$, $B(0,r) \defeq \{ x \in \R^n : |x| < r \}$.
    \item  $\mathbf{1}_I$ is the characteristic function of $I \subseteq \R$.
\end{itemize}

\section{Control of resolvent at all frequencies} \label{resolvent behavior section}

In this section, we describe the behavior of the weighted resolvent $ \langle \cdot \rangle^{-s}(G - \lambda^2)^{-1} \langle \cdot \rangle^{-s}$ for appropriate $s > 1/2$ and frequencies $\lambda$ in the upper half plane, where
$$G \defeq -c^2(x) \Delta + V(x)$$ 
with $c$ obeys \eqref{c bd above and below} and $V \in L^\infty(\R^n ; [0, \infty))$. We impose additional conditions on $c$ and $V$ depending on whether we are analyzing the resolvent near or away from zero frequency. 

\subsection{Resolvent expansion around zero frequency}
\begin{lemma} \label{low freq exp lem}
Suppose $c$ satisfies \eqref{c bd above and below} and \eqref{one minus c short range} with $\delta_0 > 2$. Let $V$ obey the same conditions as in the statement of Theorem \ref{LED thm}. Then for any $s > 1$, there exists $\kappa > 0$ so that in the set $\mathcal{O}_\kappa = \{\lambda \in \C : \imag \lambda >0, \, |\lambda| < \kappa \}$, the mapping
 \begin{equation} \label{low freq exp}
 \lambda \mapsto A_s(\lambda) \defeq \langle \cdot \rangle^{-s} \Big( (G - \lambda^2)^{-1} +  \frac{1}{2\pi} \log \Big(\frac{-i\lambda |x -y|}{2}\Big) c^{-2}\cdot \mathbf{1}_{\{2\}}(n) \Big) \langle \cdot \rangle^{-s}
 \end{equation}
 is H\"older continuous with values in the spaces of bounded operators $L^2(\R^n) \to H^2(\R^n)$, and thus extends continuously to $(-\kappa, \kappa)$. 
 \end{lemma}

\begin{proof}
    Without loss of generality, we take $1 < s < \delta_0 /2$.
Throughout the proof, $\lambda$ varies in the set $\mathcal{O}_\kappa$,
where $\kappa > 0$ will be taken sufficiently small as needed. Also recall that here $V\equiv 0$ when the dimension $n=2$.

We shall arrive at \eqref{low freq exp} by a resolvent remainder argument, which involves a Neumann series that converges for $|\lambda|$ small. In this way, $\langle \cdot \rangle^{-s}(-c^2\Delta + V - \lambda^2)^{-1} \langle \cdot \rangle^{-s}$ can be related to the resolvent expansion for $\langle \cdot \rangle^{-s}(-\Delta + V  - \lambda^2)^{-1} \langle \cdot \rangle^{-s}$, which is described in Appendices \ref{free resolv low freq appendix} and \ref{perturb resolv low freq appendix}.  A similar approach, when $1 - c$ has compact support, was taken \cite[Section 4]{sh18}.

For $\lambda \in \mathcal{O}_\kappa$, 
\begin{equation*}
\begin{gathered}
 \langle x \rangle^{-s} (-c^2 \Delta + V - \lambda^2)^{-1}  \langle x \rangle^{-s}= \langle x \rangle^{-s}  ( - \Delta + c^{-2} V  - c^{-2} \lambda^2)^{-1} \langle x \rangle^{-s} c^{-2}.
\end{gathered}
\end{equation*}
So it suffices to find a low frequency resolvent expansion for $\langle x \rangle^{-s}  ( - \Delta + c^{-2} V  - c^{-2} \lambda^2)^{-1} \langle x \rangle^{-s}$. To this end, put $V_c \defeq 1  - c^{-2}$, and observe
\begin{equation*}
\begin{split}
&(-\Delta + c^{-2} V - c^{-2} \lambda^2)(-\Delta  + c^{-2} V  -\lambda^2)^{-1} \langle x \rangle^{-s} \\
 = & \, (-\Delta +  c^{-2} V  +  \lambda^2 V_c - 
\lambda^2)(-\Delta +  c^{-2} V  -\lambda^2)^{-1}  \langle x \rangle^{-s}  \\
 = & \, \langle x \rangle^{-s}  + \lambda^2 V_c (-\Delta + c^{-2} V - \lambda^2)^{-1}  \langle x \rangle^{-s} \\
 = & \, \langle x \rangle^{-s} (I + K(\lambda) ),
\end{split}
\end{equation*}
where
\begin{equation}
    \label{K}
    K(\lambda) = \lambda^2 \langle x \rangle^{2s} V_c \langle x \rangle^{-s}  (-\Delta +  c^{-2} V  - \lambda^2)^{-1}  \langle x \rangle^{-s}.
\end{equation}
As $1 < s < \delta_0/2$, $\langle x \rangle^{2s} V_c$ is a bounded multiplication operator on $L^2(\R^n)$. This yields
\begin{equation} \label{almost to remainder formula}
\langle x \rangle^{-s} (-\Delta +  c^{-2} V -\lambda^2)^{-1} \langle x \rangle^{-s} 
= \langle x \rangle^{-s} (-\Delta +  c^{-2} V  - c^{-2}\lambda^2)^{-1} \langle x \rangle^{-s} (I + K(\lambda)).
\end{equation}

As shown in Appendices \ref{free resolv low freq appendix} ($n=2$) and \ref{perturb resolv low freq appendix} ($n \ge 3$), we have H\"older continuity $L^2(\R^n) \to H^2(\R^n)$ of
\begin{equation} \label{tilde A}
\tilde{A}_s(\lambda) \defeq \begin{cases}
\langle \cdot \rangle^{-s} (-\Delta  - \lambda^2)^{-1} \langle \cdot \rangle^{-s} + \frac{1}{2\pi}  \langle \cdot \rangle^{-s}\log\big(\frac{-i \lambda |x-y|}{2} \big)  \langle \cdot \rangle^{-s} & n= 2, \\
\langle \cdot \rangle^{-s} (-\Delta + c^{-2} V  - \lambda^2)^{-1} \langle \cdot \rangle^{-s} & n \ge 3. 
\end{cases} 
\end{equation}
Thus, \eqref{K} and \eqref{tilde A} imply that $K(\lambda)$ is also H\"older continuous on $\overline{\mathcal{O}}_\kappa$. Moreover, $\kappa$ may be taken small enough so that   $\|K(\lambda) \|_{L^2 \to L^2} < 1$ , so $I + K(\lambda)$ is invertible by Neumann series. Observe that $(I + K(\lambda))^{-1}$ is also H\"older continuous by the identity,
\begin{equation*}
(I + K(\lambda_2))^{-1} - (I + K(\lambda_1))^{-1} = (I + K(\lambda_1))^{-1} (K(\lambda_1) - K(\lambda_2)) (I + K(\lambda_2))^{-1}.
\end{equation*}
Consequently, by \eqref{almost to remainder formula}, for $\lambda \in \overline{\mathcal{O}}_\kappa$,
\begin{equation} \label{final low freq identity}
\begin{split}
& \, \langle x \rangle^{-s} (-\Delta + c^{-2} V - c^{-2} \lambda^2)^{-1} \langle x \rangle^{-s} \\
= & \, \langle x \rangle^{-s} (-\Delta + c^{-2} V  -\lambda^2)^{-1}  \langle x \rangle^{-s} (I + K(\lambda))^{-1}\\
= & \, \Big(\tilde{A}_s(\lambda) - \frac{1}{2\pi}  \langle \cdot \rangle^{-s}\log\big(\frac{-i \lambda |x-y|}{2} \big)  \langle \cdot \rangle^{-s} \mathbf{1}_{\{2\}}(n) \Big)( I - K(\lambda)(I + K(\lambda))^{-1}), 
\end{split}
\end{equation}
By \eqref{K}, $\frac{1}{2\pi}  \langle \cdot \rangle^{-s}\log\big(\frac{-i \lambda |x-y|}{2} \big)  \langle \cdot \rangle^{-s} \mathbf{1}_{\{2\}}(n) K(\lambda)$ is H\"older continuous. So \eqref{final low freq identity} may be written more succinctly:
\begin{equation*}
\langle x \rangle^{-s} (-\Delta + c^{-2} V - c^{-2} \lambda^2)^{-1} \langle x \rangle^{-s} = B_s(\lambda) - \frac{1}{2\pi}  \langle \cdot \rangle^{-s}\log\big(\frac{-i \lambda |x-y|}{2} \big)  \langle \cdot \rangle^{-s} \mathbf{1}_{\{2\}}(n), \qquad \lambda \in \overline{\mathcal{O}}_\kappa.
\end{equation*}
for some $B_s : L^2(\R^n) \to H^2(\R^n)$ H\"older continuous.

\end{proof}

\subsection{Resolvent estimate away from zero frequency}

\label{semiclassical prelim section}

We develop a resolvent estimate for $G$ away from zero frequency by rescaling semiclassically. Let $\lambda \in \C$ with $|\real \lambda| > \lambda_0$ and $0 \le \imag \lambda \le \ep_0$ for some $\lambda_0, \, \ep_0 > 0$. Make the following identifications, motivated by Section \ref{Carleman estimate section}:
\begin{equation} \label{semiclassical rescale}
\begin{gathered}
 h_0 \defeq \lambda_0^{-1}, \qquad h \defeq |\real \lambda|^{-1}, \qquad \ep \defeq \imag \lambda, \\ 
  V_L \defeq (h^2 \ep^2 -1)c^{-2}, \qquad V_S = h^2 c^{-2}V, \qquad W_L = -2 \sgn(\real \lambda) h \ep c^{-2}.
\end{gathered}
\end{equation}
We arrive at

\begin{equation*}
\begin{split}
G - \lambda^2  &= -c^2 \Delta +V - \lambda^2 \\
&= (\real \lambda)^{2} c^2 (- (\real \lambda)^{-2}\Delta +h^2c^{-2}V - c^{-2} + c^{-2}(\real \lambda)^{-2}(\imag \lambda)^2\\
&- 2i  \sgn(\real \lambda) |\real \lambda|^{-1} \imag \lambda c^{-2}) \\
&=  h^{-2} c^2(-h^2 \Delta + V_L + V_S  +iW_L),
\end{split}
\end{equation*}
with $h$ varying $(0, h_0]$ and $\ep$ in $[0, \ep_0]$.

Suppose $\ep_0$ is fixed small enough, depending on $h_0$, so that $a \defeq 1 - (\sup_{\R^n}c^{-2})h^2_0 \ep^2_0 > 0$. Then the long range potential $V_L$ possesses the properties \eqref{V long range} and \eqref{V prime long range} requested of it subsection \ref{regularity and decay of potential subsection} because
\begin{equation*}
V_L = 1 - c^{-2} - (1 - c^{-2} h^2 \ep^2) \le 1 - c^{-2} - a.
\end{equation*}
Moreover, $V_S$ obeys \eqref{V short range}, while $W_L$ satisfies \eqref{W long range} and \eqref{ratio c2 over c1}. Thus, in keeping with the notation of Section \ref{Carleman estimate section}, put
\begin{equation} \label{P}
P = P(\ep, h) \defeq -h^2 \Delta + V_L + V_S +iW_L, 
\end{equation}
so that, for $\imag \lambda > 0$,
\begin{equation} \label{resolv of G alt form}
    (G - \lambda^2)^{-1} =  h^2 P^{-1}(\ep,h)c^{-2}.
\end{equation}

For brevity of notation, put $P^{-1} = P^{-1}(\ep, h)$. The following resolvent estimate is a consequence of the semiclassical Carleman estimate proved in Section \ref{Carleman estimate section}.

\begin{lemma} \label{lap lemma}
Fix $s > 1/2, \, h_0 > 0$. Suppose $c$ obeys \eqref{c bd above and below}, \eqref{one minus c short range} for $\delta_0 > 0$, and \eqref{partialr c short range} for $\delta_1 > 1$. Let $V \in L^\infty(\R^n ; [0, \infty)$ satisfy 
\begin{equation} \label{srp V decay}
|V(x)| \le C\langle x \rangle^{-\rho}.
\end{equation}
 for some $C >0$ and $\rho > 2$. Let $\ep_0 > 0$ be sufficiently small so that $1 - (\sup_{\R^n}c^{-2})h^2_0 \ep^2_0 > 0$. There exists $C > 0$ so that for all $h \in (0, h_0]$, $\ep \in (0, \ep_0]$, and multi-indices $\alpha_1, \, \alpha_2$ with $|\alpha_1| + |\alpha_2| \le 2$,
\begin{equation} \label{lap P}
\|  \langle x \rangle^{-s} \partial^{\alpha_2}_x P^{-1} \partial^{\alpha_1}_x  \langle x \rangle^{-s} \|_{L^2(\R^n) \to L^2(\R^n) } \le e^{C/h}. 
\end{equation}
\end{lemma}

An immediate consequence of Lemma \ref{lap lemma} and \eqref{resolv of G alt form} is the following resolvent estimate for $G$ away from zero frequency.

\begin{corollary} 
Fix $s > 1/2, \, \lambda_0 > 0$. Assume $c$ and $V$ satisfy the same conditions as in the statement of Lemma \ref{lap lemma}. Let $\ep_0 > 0$ be sufficiently small so that  $1 - (\sup_{\R^n}c^{-2})\lambda^{-2}_0 \ep^2_0 > 0$. There exists $C > 0$ so that if $|\real \lambda| \ge \lambda_0$, $\imag \lambda \in (0, \ep_0]$
\begin{equation} \label{lap G}
\| \langle x \rangle^{-s} (G - \lambda^2)^{-1} c^2 \langle x \rangle^{-s}\|_{H^{-1}(\R^n) \to H^1(\R^n)} \le e^{C |\real \lambda|}.
\end{equation}
Here $H^{-1}(\R^n)$ denotes the dual space of $H^1(\R^n)$ with respect to the scalar product $\langle \cdot, \cdot \rangle_{L^2}$, with norm
\begin{equation*}
\| u \|_{H^{-1}} \defeq \sup_{0 \neq v \in H^1} \frac{ |\langle u,v \rangle_{L^2}|}{\| v\|_{H^1}}.
\end{equation*}

\end{corollary}

\begin{proof}[Proof of Lemma \ref{lap lemma}]
Without loss of generality we take $s < 1$. Over the course of the proof, $C$ denotes a positive constant whose precise value may change, but is always independent of $h$, $\ep$, and $v \in C^\infty_0(\R^n)$.

First, we treat the case $\alpha_1 = 0$. Begin from \eqref{Carleman est no remainder} in Section \ref{Carleman estimate section}. If $h \in (0, h_0]$, $\ep \in [-\ep_0, \ep_0]$, and $v \in C^\infty_0(\R^n)$, 
\begin{equation} \label{prepare to use density}
\|\langle x \rangle^{-s} v \|^2_{L^2(\R^n)} \le e^{C/h}\|\langle x \rangle^{s} (-h^2 \Delta + V_S + V_L \pm i W_L)v \|^2_{L^2(\R^n)}.  
\end{equation}
Combining this with a well known density argument, which we provide in Appendix \ref{inequalities appendix}, implies 
\begin{equation} \label{use std density}
\| \langle x \rangle^{-s} (-h^2 \Delta + V_L + V_S \pm i W_L)^{-1} \langle x \rangle^{-s}\|_{L^2(\R^n) \to L^2(\R^n)} \le e^{C/h}, \quad h \in (0, h_0], \, \ep \in [-\ep_0, \ep_0] \setminus \{0\}.
\end{equation}
Recall from standard elliptic theory that for all $f \in H^2(\R^n)$ and all $\gamma > 0$,
\begin{equation} \label{std elliptic thry}
\begin{gathered}
\| f\|_{H^2(\R^n)} \le C( \| f\|_{L^2(\R^n)} +  \| \Delta f\|_{L^2(\R^n)}), \\
\| f\|^2_{H^1(\R^n)} \le C \| f\|_{L^2(\R^n)} \|f \|_{H^2(\R^n)} \le C( \gamma^{-1} \| f\|^2_{L^2(\R^n)} + \gamma \| \Delta f\|^2_{L^2(\R^n)}).
\end{gathered}
\end{equation}
Using these with \eqref{use std density} and $-h^2\Delta = P - V_L - V_S-iW_L$, for any $f \in L^2(\R^n)$,
\begin{equation*}
\begin{split}
\| \langle x \rangle^{-s}& P^{-1} \langle x \rangle^{-s}f\|_{H^2(\R^n)} \\
&\le C ( \| \langle x \rangle^{-s} P^{-1} \langle x \rangle^{-s}f\|_{L^2(\R^n)} + \| (-\Delta) \langle x \rangle^{-s} P^{-1} \langle x \rangle^{-s}f \|_{L^2(\R^n)}) \\
&\le C ( \| \langle x \rangle^{-s} P^{-1} \langle x \rangle^{-s}f\|_{H^1(\R^n)} + h^{-2} \| \langle x \rangle^{-s} (-h^2\Delta) P^{-1} \langle x \rangle^{-s}f \|_{L^2(\R^n)}) \\
&\le C(\gamma^{-1} + h^{-2}) \| \langle x \rangle^{-s} P^{-1} \langle x \rangle^{-s}f\|_{L^2(\R^n)} + C\gamma \| \Delta \langle x \rangle^{-s} P^{-1} \langle x \rangle^{-s}f\|_{L^2(\R^n)}\\
 &+ C  h^{-2} \| f \|_{L^2(\R^n)}. 
\end{split}
\end{equation*}
The same estimate holds with $(-h^2 \Delta + V_L + V_S - iW_L)^{-1}$ in place of $P^{-1}$. Selecting $\gamma$ sufficiently small depending on $C$, yields 
\begin{equation*}
\begin{split}
\| \langle x \rangle^{-s}&(-h^2 \Delta + V_L + V_S \pm i W_L)^{-1}  \langle x \rangle^{-s}f\|_{H^2(\R^n)} \\
&\le Ch^{-2} \| \langle x \rangle^{-s} (-h^2 \Delta + V_L + V_S \pm i W_L)^{-1}  \langle x \rangle^{-s}f\|_{L^2(\R^n)} + C  h^{-2}\| f \|_{L^2(\R^n)},
\end{split}
\end{equation*}
so in view of \eqref{use std density},
\begin{equation} \label{L2 to H2 lap}
\| \langle x \rangle^{-s}(-h^2 \Delta + V_L + V_S \pm i W_L)^{-1} \langle x \rangle^{-s}f\|_{H^2(\R^n)} \le e^{C/h}\|f \|_{L^2(\R^n)}.
\end{equation}

If $|\alpha_1| > 0$, let $f \in  C^{\infty}_0(\R^n)$, and put $u = \langle x \rangle^{-s} P^{-1} \langle x \rangle^{-s} \partial^{\alpha_1}_x f$. We need to show
\begin{equation} \label{alpha1 nontriv}
\| u\|_{H^{|\alpha_2|}} \le e^{C/h} \| f\|_{L^2}, \qquad H^0 = H^0(\R^n) \defeq L^2(\R^n).   
\end{equation} 
If $|\alpha_2| = 0$, we use \eqref{L2 to H2 lap} and that the adjoint of $P$ on $L^2(\R^n)$ is $-h^2 \Delta + V_L + V_s - iW_L$. Therefore
\begin{equation*}
\begin{split}
\| u\|^2_{L^2} &= \langle u, \langle x \rangle^{-s} P^{-1} \langle x \rangle^{-s} \partial^{\alpha_1}_x f \rangle_{L^2} \\
&\le \| \partial^{\alpha_1}_x \langle x \rangle^{-s} (P^*)^{-1} \langle x \rangle^{-s} u \|_{L^2} \| f\|_{L^2} \\
&\le  e^{C/h} \| u \|_{L^2} \| f\|_{L^2}.
\end{split}
\end{equation*}
If $|\alpha_2| = 1$, we recognize that $(-h^2 \Delta 
 +V_L + V_S + iW_L)u = Pu = \langle x \rangle^{-2s} \partial^{\alpha_1}_x f + [-h^2\Delta, \langle x \rangle^{-s}] \langle x \rangle^{s} u$. Then multiply by $\overline{u}$, integrate over $\R^n$, and integrate by parts as appropriate
\begin{equation*}
\begin{split}
 \|h\nabla u \|_{L^2}^2 = - \int (V_L + V_S) |u|^2 - \int \partial^{\alpha_1}_x (\langle x \rangle^{-2s} \overline{u}) f -h^2 \int \overline{u} [\Delta, \langle x \rangle^{-s}] \langle x \rangle^{s} u. 
\end{split}
\end{equation*}
Because $[\Delta, \langle x \rangle^{-s}] \langle x \rangle^{s} = (\Delta \langle x \rangle^{-s}) \langle x \rangle^{s}  + 2 (\nabla \langle x \rangle^{-s}) \cdot \nabla \langle x \rangle^{s}$ and $\|2 (\nabla \langle x \rangle^{-s}) \cdot \nabla \langle x \rangle^{s} u\|_{L^2} \le C \| \nabla u \|_{L^2}$, we conclude, for all $\gamma > 0$,
\begin{equation*}
\begin{split}
\|h \nabla u \|^2_{L^2} &\le C  ((1 + \gamma^{-1}) \| u\|^2_{L^2} + \| f\|^2_{L^2} ) + \gamma \| h\nabla u \|^2_{L^2} \\
&\le  e^{C/h} (1 + \gamma^{-1}) \| f\|^2_{L^2} + \gamma \|h \nabla u \|^2_{L^2}.
\end{split}
\end{equation*}
 Note that $\|u\|^2_{L^2} \le e^{C/h} \|f\|^2_{L^2}$ by \eqref{alpha1 nontriv} in the case $|\alpha_2| = 0$, which we have already shown. Fixing $\gamma$ small enough, we absorb the second term on the right side into the left side, and divide by $h^2$, confirming \eqref{alpha1 nontriv} when $|\alpha_2| = |\alpha_1| = 1$.\\
\end{proof}

\subsection{H\"older continuity of the resolvent away from zero frequency} \label{Holder cont section}

\begin{lemma} \label{resolv square lemma}
Let $s > 3/2$ and $\lambda_0 > 0$. Assume $c$ and $V$ satisfy the same conditions as in the statement of Lemma \ref{lap lemma}. Fix $\ep_0 > 0$ sufficiently small so that $1 - (\sup_{\R^n} c^{-2}) \lambda^{-2}_0 \ep^2_0 > 0$. There exists $C > 0$ so that  if $|\real \lambda| \ge \lambda_0$ and $\imag \lambda \in (0, \ep_0]$,
\begin{equation} \label{lap square}
\|  \langle x \rangle^{-s}  (G - \lambda^2)^{-2}  \langle x \rangle^{-s} \|_{L^2(\R^n) \to H^1(\R^n)} \le e^{C |\real \lambda|}.
\end{equation}
\end{lemma}

The reason to show \eqref{lap square} is that it implies Lipschitz continuity of the weighted resolvent on bounded subsets of $[\lambda_0, \infty)$ or $(-\infty, -\lambda_0]$, allowing us to obtain a continuous extension of the weighted resolvent.

\begin{corollary} \label{cont ext cor}
    Under the hypotheses Lemma \ref{resolv square lemma}, the map
     \[ \lambda \to \langle x \rangle^{-s}(G-\lambda^2)^{-1}\langle x \rangle^{-s}\]
     extends continuously in the space of bounded operators from $\imag \lambda > 0$ to $(-\infty,-\lambda_0] \cup [\lambda_0,\infty).$
\end{corollary}

\begin{proof}
    For $j =1,2$ suppose that $\lambda_0 \le |\real \lambda_j| \le A$ for some $A \ge 1$, and $0 < \imag \lambda_j < \ep_0$. Let $\Gamma$ be the straight-line contour connecting $\lambda_1$ and $\lambda_2$. We use Lemma \ref{resolv square lemma} and the fundamental theorem of calculus for line integrals to calculate
    \begin{equation} \label{Holder cont high freq}
\begin{split}
\| \langle x \rangle^{-s} & ((G- \lambda^2_2)^{-1} - (G - \lambda^2_1)^{-1})\langle x \rangle^{-s} \|_{L^2 \to H^1}  \\
&= \big\| \langle x \rangle^{-s} \int_{\Gamma} \frac{d}{d\lambda} (G - \lambda^2)^{-1} d\lambda \langle x \rangle^{-s} \big\|_{L^2 \to H^1}  \\
&= 2 \big\|\int_{\Gamma} \langle x \rangle^{-s}  \lambda (G - \lambda^2)^{-2} d\lambda \langle x \rangle^{-s} \big\|_{L^2 \to H^1}  \\
&\le |\Gamma| e^{CA} =  |\lambda_2 - \lambda_1| e^{CA}.
\end{split}
\end{equation}
Thus $\langle x \rangle^{-s} (G-\lambda^2)^{-1}\langle x \rangle^{-s}$ is Lipschitz continuous on bounded subsets of $[\lambda_0, \infty)$ or $(-\infty, -\lambda_0]$ with values in the space of bounded operators $L^2(\R^n) \to H^1(\R^n)$. It therefore extends continuously to $(-\infty, -\lambda_0] \cup [\lambda_0,\infty)$.\\
\end{proof}

\begin{remark} \label{reduce three halves}
    If \eqref{Holder cont high freq} has been shown for $s > 3/2$, it holds for all $s > 1/2$ too, with possibly a smaller H\"older exponent. See \cite[Section 3]{cavo04}.
\end{remark}

\begin{proof}[Proof of Lemma \ref{resolv square lemma}]

The proof is motivated by \cite[Proof of Proposition 2.1]{cavo04}. We begin with the resolvent identity

\begin{equation} \label{begin to look at resolv square}
\begin{split}
-2 \lambda^2 \langle x \rangle^{-s} (G - \lambda^2)^{-2} \langle x \rangle^{-s}  &= 2\langle x \rangle^{-s} (G - \lambda^2)^{-1} (G - \lambda^2 + c^2 \Delta - V) (G- \lambda^2)^{-1} \langle x \rangle^{-s} \\
&= 2 \langle x \rangle^{-s} (G - \lambda^2)^{-1} \langle x \rangle^{-s} - 2\langle x \rangle^{-s} (G - \lambda^2)^{-1}  V (G - \lambda^2)^{-1} \langle x \rangle^{-s} \\
& + 2 \langle x \rangle^{-s} (G - \lambda^2)^{-1} c^2 \Delta (G - \lambda^2)^{-1} \langle x \rangle^{-s}.
\end{split}
\end{equation}
By \eqref{srp V decay} and \eqref{lap G}, the norm $L^2(\R^n) \to H^1(\R^n)$ of the second line of \eqref{begin to look at resolv square} is bounded by $e^{C|\real \lambda|}$. 

Now we examine more carefully the last line of \eqref{begin to look at resolv square}. Recall the well known formula for the Laplacian in polar coordinates,
\begin{equation*}
\Delta = \partial^2_r + (n-1)r^{-1} \partial_r + r^{-2} \Delta_{\US^{n-1}}, 
\end{equation*}
where $\Delta_{\US^{n-1}}$ is the negative Laplace Beltrami operator on $\US^{n-1}$. This implies the commutator identity
\begin{equation} \label{commutator identity}
[\Delta, r \partial_r] \defeq \Delta(r \partial_r) - r \partial_r (\Delta) = 2 \Delta. 
\end{equation}
Fix $f \in C^\infty_0(\R^n)$. Set $u \defeq (G - \lambda^2)^{-1} \langle x \rangle^{-s} f \in H^2(\R^n)$ and $V_c \defeq 1 - c^{-2}$. Let $\{u_k\}_{k=1}^\infty \subseteq C^\infty_0(\R^n)$ be a sequence converging to $u$ in $H^2(\R^n)$. Using \eqref{commutator identity},
\begin{equation} \label{approx u}
\begin{split}
2 \langle x \rangle^{-s} (G - \lambda^2)^{-1} c^2 \Delta (G - \lambda^2)^{-1} \langle x \rangle^{-s} f &= \lim_{k \to \infty} 2 \langle x \rangle^{-s} (G - \lambda^2)^{-1} c^2 \Delta u_k \\
& = \lim_{k \to \infty} \langle x \rangle^{-s} (G - \lambda^2)^{-1} c^2 [\Delta, r \partial_r] u_k,
\end{split}
\end{equation}
with convergence taken in the sense of $L^2(\R^n)$. As members of $H^{-1}(\R^n)$,
\begin{equation} \label{unpack commutator}
\begin{split}
[\Delta, r \partial_r] u_k
&=  \big( - \Delta(-r \partial_r) + r \partial_r (-\Delta) \big) u_k \\
&= \big( (- \Delta + \lambda^2 V_c + c^{-2}V - \lambda^2)(-r\partial_r) + r\partial_r (- \Delta + \lambda^2 V_c + c^{-2}V - \lambda^2) \\
& + (\lambda^2 V_c + c^{-2}V)  r \partial_r - r\partial_r (\lambda^2 V_c +  c^{-2}V) \big)u_k. 
\end{split} 
\end{equation}
Since $(- \Delta + \lambda^2 V_c + c^{-2} V - \lambda^2)^{-1} = (G - \lambda^2)^{-1} c^2$, from \eqref{approx u} and \eqref{unpack commutator} it follows that
\begin{equation} \label{more progress on resolv square}
\begin{split}
2 \langle x \rangle^{-s} (&G - \lambda^2)^{-1} c^2 \Delta (G - \lambda^2)^{-1} \langle x \rangle^{-s} f \\
 &= \langle x \rangle^{-s} (-r \partial_r) (G - \lambda^2)^{-1} \langle x \rangle^{-s} f \\
&+ \langle x \rangle^{-s} (G - \lambda^2)^{-1} c^2 ((\lambda^2V_c + c^{-2} V) r \partial_r - r\partial_r (\lambda^2 V_c + c^{-2} V)) (G - \lambda^2)^{-1} \langle x \rangle^{-s} f \\
&+\lim_{k \to \infty} \langle x \rangle^{-s} (G - \lambda^2)^{-1} c^2 r \partial_r c^{-2}  (G - \lambda^2)  u_k.
\end{split} 
\end{equation}
Our conditions on $c$ and $V$ imply $\lambda^2 V_c + c^{-2} V = O(\lambda^2 \langle r \rangle^{-\delta})$ for some $\delta > 2$. Thus, by $s > 3/2$ and \eqref{lap G}, we conclude that the operator norm $L^2(\R^n) \to H^1(\R^n)$ of both the second and third lines of \eqref{more progress on resolv square} is bounded by $e^{C|\real \lambda |}$.

It remains to control the last line of \eqref{more progress on resolv square}. In fact, we will show 
\begin{equation} \label{lim k}
\lim_{k \to \infty} \langle x \rangle^{-s} (G - \lambda^2)^{-1} c^2 r \partial_r c^{-2}  (G - \lambda^2)  u_k = \langle x \rangle^{-s} (G - \lambda^2)^{-1} c^2 r \partial_r \langle x \rangle^{-s} c^{-2} f.
\end{equation}
By $s > 3/2$ and \eqref{lap G}, the operator on the right side has norm $L^2(\R^n) \to H^1(\R^n)$ bounded by $Ce^{C|\real \lambda|}$, completing the proof of \eqref{lap square}. 

 To work toward \eqref{lim k}, fix $k \in \NN$, and let $\{ w_j \}_{j =1}^\infty \subseteq C^\infty_0(\R^n)$ be a sequence converging to $c^{-2} (G - \lambda^2)u_k$ in $L^2(\R^n)$, such that $u_k$ and the $w_j$ have support in a fixed compact subset of $\R^n$. Then $r\partial_r w_j$ converges to $r \partial_r c^{-2}(G - \lambda^2)u_k$ in $H^{-1}(\R^n)$. Thus for any $g \in L^2(\R^n)$,

 \begin{equation} \label{lim j}
 \langle g, \langle x \rangle^{-s} (G - \lambda^2)^{-1} c^2 r \partial_r c^{-2}  (G - \lambda^2)  u_k  \rangle_{L^2} = \lim_{j \to \infty}    \langle (G - \overline{\lambda}^2)^{-1} c^2 \langle x \rangle^{-s} g,  r \partial_r w_j \rangle_{L^2}.
\end{equation}
Furthermore, it holds that $r  (G - \overline{\lambda}^2)^{-1} c^2 \langle x \rangle^{-s} g \in H^1(\R^n)$. To verify this membership, it suffices to show $w \defeq \langle x \rangle (G - \overline{\lambda}^2)^{-1} c^2 \langle x \rangle^{-s} g$ belongs to $H^1(\R^n)$, since $r (G - \overline{\lambda}^2)^{-1} c^2  \langle x \rangle^{-s} g= r \langle x \rangle^{-1} w$. In turn, we have 
\begin{equation*}
\begin{gathered}
(G - \overline{\lambda}^2)  w = [G , \langle x \rangle] (G - \overline{\lambda}^2)^{-1}  \langle x \rangle^{-s}c^2 g + \langle x \rangle^{1- s} c^2 g \in L^2(\R^n), 
\end{gathered}
\end{equation*}
whence $w \in H^2(\R^n)$ by Lemma \ref{apply cutoff lem}. Continuing then from \eqref{lim j},

\begin{equation} \label{lim j 2}
\begin{split}
 \langle g,& \langle x \rangle^{-s} (G - \lambda^2)^{-1} c^2 r \partial_r c^{-2}  (G - \lambda^2)  u_k  \rangle_{L^2}\\
 &=  \lim_{j \to \infty} \langle  (\partial_r)^* r  (G - \overline{\lambda}^2)^{-1} c^2 \langle x \rangle^{-s} g,  w_j \rangle_{L^2}\\
 &= \langle  (\partial_r)^* r  (G - \overline{\lambda}^2)^{-1} c^2 \langle x \rangle^{-s} g,  c^{-2} (G - \lambda^2) u_k \rangle_{L^2}.
\end{split}
\end{equation}
Here, the adjoint of $\partial_r$ acts on $v \in C^\infty_0(\R^n)$ by $(\partial_r)^* v = (1 - n)r^{-1} v - \partial_r v$ and extends boundedly to $H^1(\R^n)$, see Lemma \ref{hardy lemma}.

We now wish to send $k \to \infty$ in \eqref{lim j 2} and conclude
\begin{equation} \label{lim k 2}
\lim_{k \to \infty} \langle g, \langle x \rangle^{-s} (G - \lambda^2)^{-1} c^2 r \partial_r c^{-2}  (G - \lambda^2)  u_k  \rangle_{L^2} = \langle  (\partial_r)^* r  (G - \overline{\lambda}^2)^{-1} c^2 \langle x \rangle^{-s} g,  c^{-2} \langle x \rangle^{-s} f \rangle_{L^2.}
\end{equation}
Since $(G -\lambda^2)u_k$ converges to $\langle x \rangle^{-s}f $ in $L^2(\R^n)$, we have \eqref{lim k 2} so long as 

Consider another sequence $\{ v_\ell\}_{\ell = 1}^\infty \subseteq C^\infty_0(\R^n)$ converging to $c^{-2} f$ in $L^2(\R^n)$. We use $s > 3/2$, \eqref{lap G}, and that $r \partial_r \langle x \rangle^{-s} v_\ell$ converges to $r \partial_r \langle x \rangle^{-s}c^{-2} f$ in $H^{-1}(\R^n)$:
\begin{equation*}
\begin{split}
\langle & (\partial_r)^* r  (G - \overline{\lambda}^2)^{-1} c^2 \langle x \rangle^{-s} g,  c^{-2}  \langle x \rangle^{-s} f \rangle_{L^2} \\
&= \lim_{\ell \to \infty} \langle (\partial_r)^* r  (G - \overline{\lambda}^2)^{-1} c^2 \langle x \rangle^{-s} g,   \langle x \rangle^{-s} v_\ell \rangle_{L^2} \\
&=  \lim_{\ell \to \infty} \langle  (G - \overline{\lambda}^2)^{-1} c^2 \langle x \rangle^{-s} g,  r \partial_r \langle x \rangle^{-s} v_\ell \rangle_{L^2} \\
&= \lim_{\ell \to \infty} \langle g , \langle x \rangle^{-s} (G - \lambda^2)^{-1} c^2 r \partial_r \langle x \rangle^{-s} v_\ell \rangle_{L^2}\\
&= \langle g , \langle x \rangle^{-s} (G - \lambda^2)^{-1} c^2 r \partial_r \langle x \rangle^{-s} c^{-2} f \rangle_{L^2}.
\end{split}
\end{equation*}
This completes the proof of \eqref{lim k} and of \eqref{lap square}.\\
\end{proof}

\section{Proof of Theorem \ref{LED thm}} \label{proof LED thm section}

In this Section, we prove Theorem \ref{LED thm}. The argument is motivated by \cite[Section 2]{cavo04}. The idea is to rewrite the wave propagators using the spectral theorem and Stone's formula. We aim to pick up time decay by integrating by parts within Stone's formula. To allow for this we first smooth out the resolvent by convolving it with an approximating identity depending on a small parameter $\ep = \ep(t)$, which tends to zero as $t \to \infty$. The H\"older regularity of the resolvent ensures that the reminder incurred from this step decays in time too.  

\begin{proof}[Proof of Theorem \ref{LED thm}]
 We give a proof of \eqref{LED sine propagator low freq}, and then conclude by pointing out the minor modifications needed to establish \eqref{LED cosine propagator low freq}.
 
We adopt the notation
\begin{equation*}
 R_s(\lambda) \defeq \langle x \rangle^{-s} (G - \lambda^2)^{-1} \langle x \rangle^{-s}. 
\end{equation*}
Let $A = A(t) = \gamma \log(t)$, for $\gamma > 0$ to be chosen in due course.  

The path to \eqref{LED sine propagator low freq} starts from Stone's formula \cite[Section 4.1]{te14}. For $f \in L^2(\R^n)$, 
\begin{equation} \label{ep to zero in Stone}
\begin{split}
\langle&x \rangle^{-s} \mathbf{1}_{[0, A^2]} G^{m - \frac{1}{2}}  \sin(tG^{1/2}) \langle x \rangle^{-s}f \\
&= \frac{1}{2\pi i} \lim_{\e \to 0^+}  \langle x \rangle^{-s}   \int^{A^2}_{0} \tau^{m-\frac{1}{2}}   \sin(t \tau^{1/2}) ((G - \tau - i\e)^{-1} - (G-\tau +i\e)^{-1}) d\tau  \langle x \rangle^{-s} f  \\
&= \frac{1}{\pi i}  \lim_{\e \to 0^+} \int^{A}_0 \lambda^{2m} \sin(t \lambda)( \langle x \rangle^{-s} (G - \lambda^2 - i\e)^{-1} \langle x \rangle^{-s}  -  \langle x \rangle^{-s} (G-\lambda^2 +i\e)^{-1}  \langle x \rangle^{-s} )  f d\lambda \\
&= \frac{1}{\pi i} \int^{A}_0 \lambda^{2m} \sin(t \lambda)( R_s(\lambda) - R_s( -\lambda)))  f d\lambda \\
&= \frac{1}{\pi i} \int^{A}_{\frac{\kappa}{2}} \lambda^{2m}  \sin(t \lambda)( (R_s(\lambda)  - R_s(-\lambda))  f d\lambda \\
&+ \frac{1}{\pi i} \int^{\frac{\kappa}{2}}_0 \lambda^{2m} \sin(t \lambda)( A_s(\lambda)  - A_s(-\lambda))  f d\lambda \\
&- \frac{1}{\pi i} \int^{\frac{\kappa}{2}}_0 \lambda^{2m} \sin(t \lambda)( \tfrac{1}{2 \pi}  \langle x \rangle^{-s} ( \log ( \tfrac{-i\lambda|x-y|}{2})  -  \log ( \tfrac{i\lambda|x-y|}{2}) )\langle y \rangle^{-s} c^{-2} \mathbf{1}_{\{2\}}(n) f d\lambda.
\end{split}
\end{equation}
Here, $\kappa$ and $B_s(\lambda)$ are as in the statement of Lemma \ref{low freq exp lem}. Between lines three and four, we can use the dominated convergence theorem, permitting us to set $\e = 0$, because, by \eqref{low freq exp}, $R_s(\lambda)$ has at worst a logarithmic singularity as $|\lambda| \to 0$. Thus we need to demonstrate decay of
\begin{equation} \label{decompose J1}
\int^{A}_{\frac{\kappa}{2}} \sin(t \lambda) F(\lambda) f d\lambda + \int^{\frac{\kappa}{2}}_0 \sin(t \lambda) F_0(\lambda)  f d\lambda  +\frac{1}{2\pi} \int^{\frac{\kappa}{2}}_0 \lambda^{2m} \sin(t \lambda) d\lambda \langle x \rangle^{-s} \langle y \rangle^{-s} c^{-2}  \mathbf{1}_{\{2\}}(n)f ,
\end{equation}
where 
\begin{equation*}
\begin{gathered}
F(\lambda) \defeq \frac{\lambda^{2m}}{\pi i} (R_s(\lambda) - R_s(-\lambda)), \\
\begin{split}
F_0(\lambda) &\defeq \frac{\lambda^{2m}}{\pi i}(( A_s(\lambda)  - A_s(-\lambda)).
\end{split}
\end{gathered}
\end{equation*}
In Appendix \ref{inequalities appendix}, we give a simple estimate utilizing integration by parts to show that for some $\nu > 0$,
\begin{equation} \label{J2 est}
\int^{\frac{\kappa}{2}}_0  \lambda^{2m}   \sin(t \lambda) d\lambda = O(t^{-\nu}).
\end{equation}
For the first and second terms of \eqref{decompose J1}, we have, by Lemma \ref{low freq exp lem} and Corollary \ref{cont ext cor}, $C > 0$ and $0 < \mu \le 1$ so that 
\begin{gather}
\| F_0(\lambda_2) - F_0(\lambda_1) \|_{L^2 \to H^1} \le C |\lambda_2 - \lambda_1|^\mu, \qquad 0 \le \lambda_1, \, \lambda_2 \le \kappa, \label{F0 Lipschitz} \\
\| F(\lambda_2) - F(\lambda_1) \|_{L^2 \to H^1} \le e^{C A} |\lambda_2 - \lambda_1|^{\mu}, \qquad \kappa/4   \le  \lambda_1, \, \lambda_2 \le 2A. \label{F Holder} 
\end{gather}

To utilize \eqref{F0 Lipschitz} and \eqref{F Holder}, let $\varphi \in C^\infty_0((-1, 1); [0,1])$ with $\int \varphi = 1$. Then, for $0 < \ep = \ep(t) \ll \kappa$, 
\begin{equation*}
\begin{gathered}
F_{0, \,\ep}(\lambda) \defeq \ep^{-1} \int_\R F_0(\lambda - \sigma) \varphi(\sigma/\ep) d\sigma,  \\
F_\ep(\lambda) \defeq \ep^{-1} \int_\R F(\lambda - \sigma) \varphi(\sigma/\ep) d\sigma,
\end{gathered}
\end{equation*}
are smooth in $(0, \infty)_\lambda$ with values varying in the space of bounded operators $L^2(\R^n) \to H^1(\R^n)$. In view of \eqref{F0 Lipschitz} and \eqref{F Holder}, 
\begin{gather}
\| F_{0, \,\ep}(\lambda) - F_0(\lambda)\|_{L^2 \to H^1}  = O(\ep^\mu), \qquad 0 \le \lambda \le \kappa, \\
\| F_\ep (\lambda) - F(\lambda)\|_{L^2 \to H^1} = O(e^{CA}\ep^\mu), \qquad \kappa \le \lambda \le A.
\end{gather}
Consequently, by adding and subtracting terms,
\begin{equation*}
\begin{split}
\big\|&\int^{A}_{\frac{\kappa}{2}} \sin(t \lambda) F(\lambda) f d\lambda + \int^{\frac{\kappa}{2}}_0 \sin(t \lambda) F_0(\lambda)  f d\lambda \big\|_{L^2 \to H^1}\\
& \le \big\|\int_0^{\frac{\kappa}{2}} \sin(t\lambda) F_{0}(\lambda) d
\lambda - \int_0^{\frac{\kappa}{2}} \sin(t\lambda) F_{0, \, \ep}(\lambda) d
\lambda \big\|_{L^2 \to H^1} \\
&+\big\|\int_{\frac{\kappa}{2}}^{A} \sin(t\lambda) F(\lambda) d
\lambda - \int_{\frac{\kappa}{2}}^{A} \sin(t\lambda) F_{ \ep}(\lambda) d
\lambda \big\|_{L^2 \to H^1} \\
& +\big\|\int_{\frac{\kappa}{2}}^A \sin(t\lambda) F_{\ep}(\lambda) d
\lambda \big\|_{L^2 \to H^1}+ \big\|\int_{0}^{\kappa} \sin(t\lambda) F_{0, \, \ep}(\lambda) d
\lambda \big\|_{L^2 \to H^1}\\
&= O(e^{CA} \ep^\mu) +\big \|\int_{\frac{\kappa}{2}}^A \sin(t\lambda) F_{\ep}(\lambda) d
\lambda \big\|_{L^2 \to H^1} + \big\|\int_{0}^{\frac{\kappa}{2}} \sin(t\lambda) F_{0, \, \ep}(\lambda) d
\lambda \big\|_{L^2 \to H^1}.
\end{split}
\end{equation*}
Integrating by parts,
\begin{equation*}
\begin{gathered}
t\int_{\frac{\kappa}{2}}^A \sin(t \lambda) F_\ep(\lambda) d\lambda = \left[-\cos(t\lambda) F_\ep(\lambda) \right]^{A}_{\frac{\kappa}{2}}  + \int_{\frac{\kappa}{2}}^A \cos(t \lambda) \frac{dF_\ep(\lambda)}{d \lambda}d\lambda, \\
t\int_0^{\frac{\kappa}{2}} \sin(t \lambda) F_{0,\ep} (\lambda) d\lambda =  \left[-\cos(t\lambda) F_{0,\ep}(\lambda) \right]_{0}^{\frac{\kappa}{2}}   + \int_0^{\tfrac{\kappa}{2}} \cos(t \lambda) \frac{dF_\ep(\lambda)}{d \lambda}d\lambda,
\end{gathered}
\end{equation*}
and invoking
\begin{equation*}
\big \|  \partial^k_\lambda F_{0, \, \ep}(\lambda)  \|_{L^2 \to H^1},  \| \partial^k_\lambda F_\ep(\lambda)  \|_{L^2 \to H^1} = O( e^{CA} \ep^{-k}), \qquad k \in \{0, 1\},
\end{equation*}
which follows from the definitions of $F_{0, \ep}$ and $F_\ep$, we conclude
\begin{equation*} 
\big\|\int^{A}_{\frac{\kappa}{2}} \sin(t \lambda) F(\lambda) f d\lambda + \int^{\frac{\kappa}{2}}_0 \sin(t \lambda) F_0(\lambda)  f d\lambda \big\|_{L^2 \to H^1} = O(e^{CA} ( \ep^\mu + t^{-1} \ep^{-1})).
\end{equation*}
 Finally, take $\ep(t) = t^{-1/2}$, $t \gg 1$. Since $A(t) =  \gamma \log t$ for $ \gamma > 0$ to be chosen, 
 \begin{equation} \label{J1 est final step}
\big\|\int^{A}_{\frac{\kappa}{2}} \sin(t \lambda) F(\lambda) f d\lambda + \int^{\frac{\kappa}{2}}_0 \sin(t \lambda) F_0(\lambda)  f d\lambda \big\|_{L^2 \to H^1} = O(t^{C\gamma}  ( t^{-\mu/2} + t^{-1/2})).
 \end{equation}
 Thus, fixing $\gamma$ sufficiently small, \eqref{LED sine propagator low freq} follows in view of \eqref{J2 est} and \eqref{J1 est final step}.
 
 The proof of \eqref{LED cosine propagator low freq} follows the same steps. The only difference is that an extra factor of $\lambda$ appears after the change of variable between lines two and three of \eqref{ep to zero in Stone}. The integrands in \eqref{decompose J1}. But this does not hinder reaching an $O_{L^2 \to H^1}(t^{-\nu})$ bound as in \eqref{J1 est final step}.  \\
\end{proof}

\section{Proof of Theorem \ref{brief decay}}
\label{proof of prelim decay appendix}

In this section we prove Theorem \ref{brief decay} using Theorem \ref{LED thm}, and along the way establish statements of decay that depend on the amount of regularity the initial conditions possess with respect to $G$. 

\begin{proof}[Proof of Theorem \ref{brief decay}]

Initially, take $s > 1$. For $\eta > 0$, let $D(G^{\frac{1}{2} + \eta})$ and $D(G^{\eta})$ denote the domains of the operators $G^{\frac{1}{2} + \eta}$ and $G^\eta$, respectively. Suppose 
\begin{equation} \label{technical hypotheses}
u_0 \in D(G^{\frac{1}{2} + \eta}), \quad u_1\in D(G^{\eta}), \quad \langle x \rangle^{s} u_0 , \langle x \rangle^{s}u_1 \in L^2(\R^n).
\end{equation}
Let $u(\cdot, t)$ be the solution to \eqref{wave eqn} given by the spectral theorem: 
\begin{equation*}
u(\cdot, t) = \cos(t G^{\frac{1}{2}}) u_0 + \frac{\sin(t G^{\frac{1}{2}})}{G^{\frac{1}{2}}} u_1.
\end{equation*}
Multiply $u(\cdot, t)$ and $\partial_t u(\cdot, t)$ by $\langle x \rangle^{-s}$ and decompose as follows 
\begin{equation*}
\begin{split}
u(\cdot, t)&=  (\cos(t G^{\frac{1}{2}}) u_0 +  \tfrac{\sin(t G^{\frac{1}{2}})}{G^{\frac{1}{2}}} u_1) \\
&=  \mathbf{1}_{[0, A^2(t)]} (\cos(t G^{\frac{1}{2}})  \langle x \rangle^{-s} (\langle x \rangle^{s} u_0) + \tfrac{\sin(t G^{\frac{1}{2}})}{G^{\frac{1}{2}}} \langle x \rangle^{-s} (\langle x \rangle^{s} u_1)) \\
&+  \mathbf{1}_{(A^2(t), \infty)}(G)(\tfrac{\cos(t G^{\frac{1}{2}})}{G^{\frac{1}{2} + \eta}} (G^{\frac{1}{2} + \eta}u_0)+ \tfrac{\sin(t G^{\frac{1}{2}})}{G^{\frac{1}{2} + \eta}}  (G^\eta u_1)) \\
&\qefed  u_{\le A^2(t)} + u_{> A^2(t)}. 
\end{split}
\end{equation*}
Similarly,  
\begin{equation} \label{long dtu calc}
\begin{split}
 \partial_t u(\cdot,t) & =  (-G^{\frac{1}{2}} \sin(t G^{\frac{1}{2}}) u_0 +\cos(t G^{\frac{1}{2}})  u_1) \\
& = \mathbf{1}_{[0, A^2(t)]}(G) (-G^{\frac{1}{2}} \sin(t G^{\frac{1}{2}}) u_0 +\cos(t G^{\frac{1}{2}})  u_1) \\
&+ \mathbf{1}_{(A^2(t), \infty)}(G)(-G^{\frac{1}{2}} \sin(t G^{\frac{1}{2}}) u_0 +\cos(t G^{\frac{1}{2}})  u_1)\\
& =  \mathbf{1}_{[0, A^2(t)]}(G) (-G^{\frac{1}{2}} \sin(t G^{\frac{1}{2}})  \langle x \rangle^{-s}  ( \langle x \rangle^{s}  u_0)  + \cos(t G^{\frac{1}{2}}) \langle x \rangle^{-s} (\langle x \rangle^{s} u_1)) \\
&+ \mathbf{1}_{(A^2(t), \infty)}(G)(-\tfrac{\sin(t G^{\frac{1}{2}})}{G^{\eta}} (G^{\frac{1}{2} + \eta} u_0) + \tfrac{\cos(t G^{\frac{1}{2}})}{G^{\eta}}(G^\eta u_1)) \\
&\qefed \partial_tu_{\le A^2(t)} + \partial_t u_{> A^2(t)}. 
\end{split}
\end{equation}
Therefore, under \eqref{technical hypotheses}, by \eqref{LED cosine propagator low freq} and \eqref{LED sine propagator low freq}, for $ |t| \gg 1$,
\begin{equation*}
\begin{gathered}
\| \langle x \rangle^{-s} u_{\le A^2(t)} \|_{H^1} = O(|t|^{-\nu})( \| \langle x \rangle^{s}u_0 \|_{L^2} + \| \langle x \rangle^{s}u_1 \|_{L^2}), \\
\| \langle x \rangle^{-s} \partial_t u_{\le A^2(t)} \|_{L^2} = O(|t|^{-\nu})( \| \langle x \rangle^{s} u_0 \|_{L^2} + \| \langle x \rangle^{s}u_1 \|_{L^2}).
\end{gathered}
\end{equation*}

On the other hand, under \eqref{technical hypotheses},
 \begin{equation*}
 \begin{split}
 \| & \partial_t u_{> A^2(t)} \|_{L^2} \\
 &\le  \| \mathbf{1}_{(A^2(t), \infty)}(G)(\tfrac{\cos(t G^{\frac{1}{2}})}{G^{\eta}}  \|_{L^2 \to L^2} \|G^{\frac{1}{2} + \eta}u_0\|_{L^2} + \| \mathbf{1}_{(A^2(t), \infty)}(G) \tfrac{\sin(t G^{\frac{1}{2}})}{G^{\eta}}\|_{L^2 \to L^2} \| G^{\eta}u_1\|_{L^2} \\
 &= O((\log |t|)^{-2\eta}) (\|G^{\frac{1}{2} + \eta}u_0\|_{L^2} + \| G^{\eta}u_1\|_{L^2}),
 \end{split} 
 \end{equation*}
where we used
 \begin{equation*} 
 \| f(G) \|_{L^2 \to L^2} = \| f \|_{L^\infty}, \qquad f \text{ a bounded Borel function on $\R$}. 
 \end{equation*}

Furthermore, since 
 \begin{equation} \label{sqrt G norm characterize}
 \| \nabla v \|_{L^2} \le \|\sqrt{G} v \|_{L^2_c} \le (\sup_{\R^n} c^{-2}) \| \sqrt{G} v\|_{L^2}
 \end{equation}
we have
 \begin{equation*}
 \begin{split}
 \| & u_{> A^2(t)} \|_{H^1} \\
 &=  O\big( \| \mathbf{1}_{(A^2(t), \infty)}(G)(\tfrac{\cos(t G^{\frac{1}{2}})}{G^{\frac{1}{2}+ \eta}}  \|_{L^2 \to H^1}  \big) \| G^{\frac{1}{2} + \eta}u_0\|_{L^2} \\
 &+ O \big(\|\mathbf{1}_{(A^2(t), \infty)}(G) \tfrac{\sin(t G^{\frac{1}{2}})}{G^{\frac{1}{2} + \eta}} \|_{L^2 \to H^1} \big) \| G^{\eta} u_1 \|_{L^2} \\
  &= O\big( \| \mathbf{1}_{(A^2(t), \infty)}(G)(\tfrac{\cos(t G^{\frac{1}{2}})}{G^{\frac{1}{2} + \eta}}  \|_{L^2 \to L^2} + \| \mathbf{1}_{(A^2(t), \infty)}(G)(\tfrac{\cos(t G^{\frac{1}{2}})}{G^{\eta}}  \|_{L^2 \to L^2}  \big) \| G^{\frac{1}{2} + \eta}u_0\|_{L^2}  \\
 &+ O \big(\|\mathbf{1}_{(A^2(t), \infty)}(G) \tfrac{\sin(tG^{\frac{1}{2}})}{G^{\frac{1}{2} + \eta}} \|_{L^2 \to L^2}  + \|\mathbf{1}_{(A^2(t), \infty)}(G) \tfrac{\sin(t G^{\frac{1}{2}})}{G^{\eta}} \|_{L^2 \to L^2} \big) \| G^{\eta} u_1 \|_{L^2} \\
 &= O((\log |t|)^{-2\eta}) (\| G^{\frac{1}{2} + \eta} u_0\|_{L^2} + \|G^{\eta}u_1\|_{L^2}),
 \end{split} 
 \end{equation*}
 
 Summarizing our conclusions under \eqref{technical hypotheses}:
\begin{gather}
\| \langle x \rangle^{-s} u_{\le A^2(t)}\|_{H^1}  + \| \langle x \rangle^{-s} \partial_t u_{\le A^2(t)} \|_{L^2} = O(|t|^{-\nu}) (\| \langle x \rangle^{s}u_0 \|_{L^2} + \| \langle x \rangle^{s}u_1 \|_{L^2}),\label{power decay le A} \\
 \| u_{> A^2(t)} \|_{H^1} + \| \partial_t u_{> A^2(t)} \|_{L^2} =O( (\log |t|)^{-2\eta })(\| G^{\frac{1}{2} + \eta}u_0\|_{H^2} + \|G^\eta u_1\|_{H^1}). \nonumber
\end{gather}
If $\eta = 1/2$, this implies \eqref{prelim decay} for $s > 1$. Finally, to show \eqref{prelim decay} for any $s > 0$, interpolate between \eqref{power decay le A} and the trivial bound 
\begin{equation*}
\begin{split}
\| \nabla u_{\le A^2(t)}\|_{L^2}  + \|\partial_t u_{\le A^2(t)}\|_{L^2} &= O(\| \cos(t \sqrt{G}) \sqrt{G} u_0 \|_{L^2} + \| \sin(t \sqrt{G}) u_1\|_{L^2}) \\
&= O(1) (\| \sqrt{G} u_0 \|_{L^2} + \|u_1\|_{L^2}).
\end{split}
\end{equation*}
\end{proof}

\section{Semiclassical Carleman estimate} \label{Carleman estimate section}

In this section we give semiclassical estimates that lead to the proof of Lemma \ref{lap lemma}.

\subsection{Regularity and decay of the potential} \label{regularity and decay of potential subsection}

We study semiclassical Schr\"odinger operators 
\begin{equation}   \label{P highd}
 P(\ep ,h) \defeq -h^2 \Delta + V(x ; \ep, h)  : L^2(\R^n) \to L^2(\R^n).
 \end{equation}
Here, we suppose $\ep$ and $h$ vary in  $[-\ep_0, \ep_0]$ and $(0,h_0]$, respectively, for some $\ep_0, h_0 > 0$. The potential $V(x; \ep, h)$ may depend on $\ep$ and $h$ in a manner we specify below, and may be complex-valued, with certain restrictions on its imaginary part.

We are interested in imposing minimal regularity and decay conditions on $V$ such that we can obtain an optimal semiclassical Carleman estimate for \eqref{P highd}. To this end, fix $a > 0$, along with
\begin{equation*}
\begin{gathered}
p : [0, \infty) \to (0, \infty) \text{ decreasing to zero}, \\
m(r) : [0, \infty) \to (0,1] \text{ satisfying }(r + 1)^{-1}m(r) \in L^1((0, \infty),dr) \text{ and }\lim_{r \to \infty}m(r) = 0,
\end{gathered}
\end{equation*}
and 
\begin{equation*}
\mu \text{ a nonnegative, finite, compactly supported Borel measure on $(0,\infty)$.}
\end{equation*}

We require that the real part of $V$ to belongs to $L^\infty(\R^n)$ for all $h \in (0, h_0]$, $\ep  \in [-\ep_0,\ep_0]$, and decomposes into short and long range parts:
\begin{equation} \label{real part}
\real V(\cdot \, ; \ep, h) = V_S(\cdot \, ; \ep, h) + V_L(\cdot \, ; \ep, h).
\end{equation}
For the short range part $V_S$, there exist $c_V, \, \delta_0 > 0$ so that
\begin{equation}
|V_S(x ; \ep, h)| \le c_V h ( r + 1)^{-1 -\delta_0}, \qquad h \in (0, h_0], \, \ep  \in [-\ep_0,\ep_0], \, x \in \R^n.  \label{V short range}
\end{equation}
As for the long range part $V_L$,
\begin{equation}
V_L(x; \ep, h) + a \le p(r), \qquad h \in (0, h_0], \, \ep  \in [-\ep_0,\ep_0], \, \in x \in \R^n.\label{V long range}
 \end{equation}
In addition, for each $h \in (0, h_0]$, $\ep \in [-\ep_0, \ep_0]$, and $\theta \in \US^{n-1}$ we require the mapping
\begin{equation*}
 (0, \infty) \ni r \mapsto V_L(r, \theta; \ep ,h) 
\end{equation*}
to be of locally bounded variation; for any interval $I$ whose closure lies in $(0, \infty)$, the total variation of $V_L(\cdot, \theta; \ep ,h)$ over $I$ should be uniformly bounded with respect to $h$, $\ep$, and $\theta$. We also assume that the associated measure $dV_L(\cdot, \theta; \ep, h)$--reviewed in subsection \ref{BV review section}--satisfies the bound
 \begin{equation}\label{V prime long range}
  \int_E dV_L( \cdot, \theta; \ep, h) \le c_V \int_E (r +1)^{-1} m(r) dr + \int_E \mu,
\end{equation}
for every bounded Borel set $E \subseteq (0, \infty)$.

On the imaginary part of $V$, we also impose a decomposition into short and long range terms,
\begin{equation} \label{imag part}
\imag V(\cdot \, ; \ep, h) = W_S(\cdot \, ; \ep, h) + W_L(\cdot \, ; \ep, h).
\end{equation}
The long range term $W_L$ needs to have a fixed sign, with constants $0 < c_1 = c_1(\ep, h) \le c_2 = c_1(\ep, h)$ so that for all $h \in (0, h_0]$, $\ep  \in [-\ep_0,\ep_0]$, and $x \in \R^n$,
\begin{gather} 
c_1(\ep, h) \le |W_L(x ; \ep, h)| \le c_2(\ep, h) \le c_V, \label{W long range} \\
 c^{-1}_1 c_2 \le c_V.\label{ratio c2 over c1}
\end{gather}
The short range term $W_S$ has the same fixed sign as $W_L$, and
\begin{equation}
  |W_S(x ; \ep, h)| \le c_Vh (r + 1)^{-1 -\delta_0}. \label{W short range} 
\end{equation}

Since the conditions on $V$ are technical, for intuition we encourage the reader to keep in the mind the prototypical example in which $W_L(x; \ep, h) \equiv \ep$, the other terms are independent of $\ep$, and $V_L(x;h) = \tilde{V}_L(x;h) - a$ for $\tilde{V}_{L}$ long range and decaying to zero as $r \to \infty$, i.e.,  
\begin{equation*}
V(x; \ep, h) = \tilde{V}_L(x;h) + V_S(x;h) + W_S(x;h) - a - i \ep.
\end{equation*}
In this case we could think of $a + i\ep$ as playing the role of a spectral parameter.
 
\subsection{Review of BV} \label{BV review section}

We recall well-known properties of functions of bounded variation, which facilitate the proof of our Carleman estimate. Proofs may be found in \cite[Appendix B]{dash23}.

Let $I$ be a (possibly infinite) open interval $I \subseteq \R$. Suppose  $f : I \to \C$ is of locally bounded variation, meaning each of $\real f$ and $\imag f$ is the difference of two increasing functions. For all $x \in I$, put
\begin{equation} \label{LRA}
 f^L(x) \defeq \lim_{\delta \to 0^+}f(x-\delta), \qquad  f^R(x) \defeq \lim_{\delta \to 0^+}f(x+\delta), \qquad f^A(x) \defeq (f^L(x) + f^R(x))/2.
\end{equation}
Recall $f$ is differentiable Lebesgue almost everywhere, so $f(x)= f^L(x) = f^R(x) = f^A(x)$ for almost all $x \in I$.

We may decompose $f$ as 
\begin{equation} \label{decompose f}
f = f_{r, +} - f_{r, -} + i( f_{i,+} - f_{i,-}),
\end{equation}
where the $f_{\sigma,\pm}$, $\sigma \in \{r, i\}$, are increasing functions on $I$. Each $f^R_{\sigma,\pm}$ uniquely determines a regular Borel measure $\mu_{\sigma,\pm}$ on $I$ satisfying $\mu_{\sigma, \pm}(x_1, x_2] = f^R_{\sigma, \pm}(x_2) -  f^R_{\sigma, \pm}(x_1)$, see \cite[Theorem 1.16]{fo}. We put
\begin{equation} \label{df}
    df \defeq \mu_{r, +} - \mu_{r, -} + i( \mu_{i,+} - \mu_{i,-}),
\end{equation}
 which is a complex measure when restricted to any bounded Borel subset of $I$. For any subinterval $(a, b] \subseteq I$,
 
 \begin{equation} \label{ftc}
 \begin{gathered}
  \int_{(a,b]}df = f^R(b) - f^R(a).
  \end{gathered}
 \end{equation}

\begin{proposition}[product rule] \label{prod rule bv prop}
Let $f, \, g : I \to \C$ be functions of locally bounded variation. Then
\begin{equation}\label{prod}
 d(fg) = f^A dg + g^A df
\end{equation}
as measures on a bounded Borel subset of $I$.
\end{proposition}


\begin{proposition}[chain rule] \label{chain rule bv prop}
Let $f : I \to \R$ be continuous and have locally bounded variation. Then, as measures on a bounded Borel set of $I$,
\begin{equation} \label{chain rule continuous}
    d(e^f) = e^{f} df.
\end{equation}
\end{proposition}
 
\subsection{Preliminary calculations} \label{preliminary calculations subsection}

We set the stage for proving the Carleman estimate by means of the so-called energy method, which is a frequently chosen strategy for establishing Carleman estimates in low regularity (see, e.g., \cite{cavo02, da14, gash22a, ob24}). Throughout, we take $h \in (0, h_0]$, $\ep \in [-\ep_0, \ep_0]$, and assume for all such $h$ and $\ep$, that  $V(\cdot \, ; \ep, h)$ obeys \eqref{real part} through \eqref{W short range}. Let $P(\ep, h)$ be given by \eqref{P highd}. 

We work in polar coordinates, beginning from the well known identity
\begin{equation*}
    r^{\frac{n-1}{2}}(- \Delta) r^{-\frac{n-1}{2}} = -\partial^2_r + r^{-2} \Lambda,
\end{equation*}
where 
\begin{equation} \label{Lambda}
    \Lambda \defeq -\Delta_{\US^{n-1}} + \frac{(n-1)(n-3)}{4},
\end{equation}
and $\Delta_{\US^{n-1}}$ denotes the negative Laplace-Beltrami operator on $\US^{n-1}$. Let $\varphi(r; h)$ be a soon-to-be-constructed phase on $(0, \infty)$ which depends on $h$ but is independent of $\ep$. We ask that that $\varphi$ and $\varphi'$ are nonnegative and locally absolutely continuous, $(\varphi')^2 - h \varphi''$ has locally bounded variation, and $\varphi(0 ; h) = 0$. Using $\varphi$, we form the conjugated operator
\begin{equation} \label{conjugation}
\begin{split}
  P_\varphi(\ep, h) &\defeq e^{\frac{\varphi}{h}} r^{\frac{n-1}{2}} P(\ep, h) r^{-\frac{n-1}{2}} e^{-\frac{\varphi}{h}}\\
  &= -h^2\partial^2_r + 2h \varphi' \partial_r + h^2r^{-2} \Lambda +  V -(\varphi')^2 + h\varphi'' .
 \end{split}
\end{equation}

Let
\begin{equation} \label{function space u}
 u \in e^{\varphi/h} r^{(n-1)/2} C^\infty_0(\R^n).
\end{equation}
Define a spherical energy functional $F[u](r)$,
\begin{equation} \label{F}
    F(r) = F[u](r) \defeq \|hu'(r, \cdot)\|^2 - \langle (h^2 r^{-2} \Lambda +  V_L   -(\varphi')^2  + h \varphi'')u(r, \cdot), u(r, \cdot) \rangle,
\end{equation}
where $\| \cdot \|$ and $\langle \cdot, \cdot \rangle$ denote the norm and inner product on $L^2(\mathbb{S}_\theta^{n-1})$, and complex conjugation in $\langle \cdot, \cdot \rangle$ takes place in the first argument. For a weight $w(r)$ which is independent of $h$ and $\ep$, absolutely continuous, nonnegative, increasing, and bounded, we compute the distributional derivative of $wF$ on $(0, \infty)$. The most delicate term of \eqref{F}  to differentiate is $r \mapsto w(r) \int_{\US^{n-1}} V_L(r, \theta) |u(r, \theta)|^2 dr $. In \cite[Appendix A]{llst24} we show this mapping has locally bounded variation and its distributional derivative is
\begin{equation} \label{technical formula}
  \begin{split}
  C^\infty_0(0, \infty) \ni \phi &\mapsto \int_0^\infty  w(r) \phi(r) \int_{\US^{n-1}} V(r, \theta) 2 \real(\overline{u}u')  d\theta dr\\
   &+\int_0^\infty \int_{\US^{n-1}}  w'(r) \phi(r) |u(r, \theta)|^2  dr d\theta\\
  &+ \int_{\US^{n-1}}  \int_0^\infty  w(r) \phi(r) |u(r, \theta)|^2 dV(r,\theta) d\theta
  \end{split}
  \end{equation} 
In the subsequent calculation we denote the last term of \eqref{technical formula} by $\int_{\US^{n-1}} |u(r, \theta)|^2 w(r) dV_L(r, \theta)d\theta$. We have

\begin{equation} \label{deriv wF}
\begin{split}
    d(wF) &=  wdF +  w'F\\
    &= w(-2\real \langle (-h^2u'' + h^2r^{-2}\Lambda + V_L - (\varphi')^2 + h \varphi'' )u  , u' \rangle \\
     &+2 h^2 r^{-3} \langle \Lambda u, u \rangle +  \|u\|^2 d((\varphi')^2  - h \varphi''))  - \int_{\US^{n-1}} |u(r, \theta)|^2 w(r) dV_L(r, \theta)  \\
     &+ (\|hu'\|^2 - \langle h^2 r^{-2} \Lambda  u, u \rangle + ((\varphi')^2  - h \varphi'' - V_L) \|u\|^2) w' \\
    &= -2 w \real \langle P_\varphi(\ep, h) u, u' \rangle + 2w \real \langle V_S u, u' \rangle + 2w \imag \langle (W_S + W_L ) u,u'\rangle \\
    &+ \|hu'\|^2 (4h^{-1}w \varphi' + w') +\langle h^2r^{-2}(-\Delta_{\US^{n-1}} + 4^{-1}(n-1)(n-3)) u,u\rangle (2wr^{-1} - w') \\
    &+  \|u\|^2 d(w((\varphi')^2  - h \varphi'')) - \| u\|^2 V_L w'   - \int_{\US^{n-1}} |u(r, \theta)|^2 w(r) dV_L(r, \theta)d\theta. 
    \end{split}
\end{equation}


Since $w \varphi' \ge 0$, we can discard the term  $4h^{-1}w \varphi'  \|hu'\|^2$ when finding a lower bound for \eqref{deriv wF}. We can also discard the term involving $-\Delta_{\US^{n-1}}$ if
\begin{equation}
q(r) \defeq 2wr^{-1} - w' \ge 0 \label{w w prime est},
\end{equation}
which we shall arrange. Using also $4^{-1}(n-1)(n-3) \ge -4^{-1}$, \eqref{V short range}, \eqref{V prime long range}, and \eqref{W short range}, we find, for all $\gamma > 0$,
\begin{equation} \label{deriv wF 2}
\begin{split}
    d(wF) &\ge -\tfrac{4w^2}{h^2w'} \|P_\varphi(\ep, h) u\|^2 - 2^{-1}h^{-1} w^2 \| |W_L|^{1/2} hu'\|^2 - 2h^{-1}  \|  |W_L|^{1/2}u\|^2 \\
    &+  \|hu'\|^2 w' ( \tfrac{3}{4} - \tfrac{2 \gamma c_V w}{(r +1)^{1 + \delta_0}w'})  \\
    &+ \|u\|^2 \big( d(w((\varphi')^2 - h\varphi'' )) -V_L w' - c_V(r+1)^{-1}m w -  w\mu - \tfrac{2 c_V w}{\gamma (r + 1)^{1+\delta_0}} -\tfrac{h^2 q}{4r^2}  \big). \\
    \end{split}
\end{equation}
Note that, because $4^{-1}(n-1)(n-3) \ge 0$ except in dimension two, the last term in line three of \eqref{deriv wF 2} can be disregarded except in dimension two. This is relevant to Remark \ref{optimal ext weight remark} below.

By \eqref{W long range} and \eqref{prod}, we bound from above the term $w ^2\| |W_L|^{1/2} hu'\|^2$:
\begin{equation} \label{est hu prime square 1}
\begin{split}
 w^2 \|& |W_L|^{1/2} hu'\|^2 \le c_2 w^2 \| hu'\|^2   \\
 &= \real \big( ( c_2 w^2  \langle hu', hu \rangle)'- c_2 w \langle  h^2u'', u \rangle  - 2 c_2 w w' \langle hu', hu \rangle \big).
 \end{split}
 \end{equation}
 The last two terms in the second line of \eqref{est hu prime square 1} may be estimated as follows. From \eqref{W long range}, \eqref{ratio c2 over c1}, and \eqref{conjugation},
 \begin{equation} \label{est hu prime square 2}
 \begin{gathered}
 \begin{split}
 -2 c_2 w w' \real \langle hu', hu \rangle &\le (w c_1^{1/2}  \|hu'\|)(2c^{-1}_1c_2 w' c_1^{1/2} \|hu\|),  \\
 & \le \tfrac{1}{4} w^2 \| |W_L|^{1/2} hu'\|^2 + 4c^2_V h^2 (w')^2  \| |W_L|^{1/2} u\|^2, 
 \end{split} \\
 \begin{split}
- c_2 w &\real \langle  h^2u'', u \rangle \\
&= c_2 w \real \big(\langle (P_\varphi(\ep, h)  -2h \varphi' \partial_r -h^2 r^{-2} \Lambda - V + (\varphi')^2 - h \varphi'')u ,u  \rangle \big)  \\
 &\le  \tfrac{1}{4} w^2 \| |W_L|^{1/2} hu'\|^2 + \tfrac{c^2_V w^2}{2} \|P_\varphi(\ep, h)u\|^2 \\
 &+ \big( \tfrac{1}{2} + \tfrac{h^2c_V w}{4r^2} + c_V w\|\real V\|_{L^\infty} + (4c^2_V + c_Vw)(\varphi')^2 + h c_V w |\varphi''| \big) \||W_L|^{1/2} u\|^2. 
 \end{split}
 \end{gathered}
\end{equation} 
Therefore, combining \eqref{deriv wF 2}, \eqref{est hu prime square 1}, and \eqref{est hu prime square 2}, we have, for all $h \in (0, h_0]$, $\ep  \in [-\ep_0, \ep_0]$, and $\gamma > 0$.
\begin{equation} \label{deriv wF 3}
\begin{split}
    d(wF) &\ge -\big(\tfrac{4w^2}{h^2w'} + \tfrac{c^2_V w^2}{2} \big)\|P_\varphi(\ep, h) u\|^2 \\
    &-( \real ( c_2 w^2  \langle hu', hu \rangle))' \\
    &- h^{-1} \| |W_L|^{1/2} u\|^2  \big( \tfrac{5}{2} + \tfrac{ h^2 c_V w}{4r^2} + c_V w\|\real V\|_{L^\infty} \\
    &+ (4c^2_V + c_Vw)(\varphi')^2 + h c_V w |\varphi''| + 4c^2_V h^2 (w')^2\big) \\
   &+  \|hu'\|^2w' ( \tfrac{3}{4} - \tfrac{2 \gamma c_V w}{(r +1)^{1 + \delta_0}w'})  \\
   &+ \|u\|^2 \big( d(w((\varphi')^2 - h\varphi'' )) -V_L w' - c_V(r+1)^{-1}m w -  w \mu - \tfrac{2 c_V w}{\gamma (r + 1)^{1+\delta_0}} -\tfrac{h^2 q}{4r^2}  \big).    
    \end{split}
\end{equation}

\subsection{Construction of the phase and weight} \label{phase and weight section}

To produce a Carleman estimate from \eqref{deriv wF 3}, it is essential that we specify $w$ and $\varphi$ precisely, in order that the last two lines line of \eqref{deriv wF 2} have a good lower bound. We thus proceed with designing the appropriate weight and phase.

First we specify several constants. Fix $s$ such that
\begin{equation} \label{s}
1< 2s < 1 + \delta_0.
\end{equation}
Choose $R_0 \ge 1$ large enough so that the measure $\mu$ in \eqref{V prime long range} is supported in $(0, R_0]$, and so that by \eqref{V long range} and $\lim_{r \to \infty}m(r) = 0$, we have for all $h \in (0, h_0]$, $\ep  \in [-\ep_0,\ep_0]$, and $(r, \theta) \in (R_0, \infty) \times \US^{n-1}$,
\begin{equation} \label{R0}
V_L(r, \theta; \ep, h) + a, \, c_Vm(r) , \, \tfrac{16c^2_V \langle r \rangle^{2s}}{(r + 1)^{1+\delta_0}}  \le  \frac{a}{4}.
\end{equation}
 
The weight function we utilize is
\begin{gather} 
w = \begin{cases} r^2 & 0 < r \le M \\
 M^2e^{ \int_{M}^r \max(\kappa_1 (r'+1)^{-1} m(r'), \, \kappa_2 \langle r' \rangle^{-2s}) dr'} & r > M  \\
\end{cases}, \label{w} \\ 
  w' = 2r \mathbf{1}_{(0,M]} + \max( \kappa_1 (r+1)^{-1} m(r), \, \kappa_2 \langle r \rangle^{-2s}) w \mathbf{1}_{(M, \infty)},
\label{w prime}  
\end{gather}
where $M \ge 2R_0 > 2$ and $\kappa_1 \ge 0, \, \kappa_2 \ge 1$ are to be fixed, independent of $h$ and $\ep$, over the course of the proof of Lemma~\ref{goal est lemma} below. Notice that, in the sense of measures,
\begin{equation} \label{w over dw}
\tfrac{w}{w'} \le \tfrac{r}{2} \mathbf{1}_{(0, M]} + \tfrac{ \langle r \rangle^{2s}} {\kappa_2} \mathbf{1}_{(M, \infty)},
\end{equation}
so into the fourth line of \eqref{deriv wF 3},
\begin{equation*}
\begin{split}
 \tfrac{3}{4} - \tfrac{2 \gamma c_V w}{w'(r +1)^{1 + \delta_0}}  &\ge \tfrac{3}{4} - 2 \gamma c_V.
\end{split}
\end{equation*}
Thus we fix 
\begin{equation} \label{gamma}
\gamma = (8c_V)^{-1}.
\end{equation}
Hence \eqref{deriv wF 3} implies,
\begin{equation} \label{deriv wF 4}
\begin{split}
    d(wF) &\ge -\big(\tfrac{4w^2}{h^2w'} + \tfrac{c^2_V w^2}{2} \big)\|P_\varphi(\ep, h) u\|^2 \\
    &-d( \real ( c_2 w^2  \langle hu', hu \rangle)) \\
    &- h^{-1} \| |W_L|^{1/2} u\|^2  \big( \tfrac{5}{2} + \tfrac{ h^2 c_V w}{4r^2} + c_V w\|\real V\|_{L^\infty} \\
    &+ (4c^2_V + c_Vw)(\varphi')^2 + h c_V w |\varphi''| + 4c^2_V h^2 (dw)^2\big) \\
    &+ \tfrac{1}{2} \|hu'\|^2 ( r \mathbf{1}_{(0,M]} +  \langle r \rangle^{-2s} \mathbf{1}_{(M, \infty)})  \\
     &+ \|u\|^2 \big( d(w((\varphi')^2 - h\varphi'' )) -V_L w' - c_V(r+1)^{-1}m w -  w \mu - \tfrac{16 c_V^2 w}{ (r + 1)^{1+\delta_0}} -\tfrac{h^2 q}{4r^2}  \big). 
    \end{split}
\end{equation}

Continuing, define the function $\psi(r)$, independent of $h$ and $\ep$, by 
\begin{gather}
\psi(r) \defeq \begin{cases} 
p(0) + (1 + 16c_V) c_V  + \mu(0, r]  & 0 < r \le R_0 \\
\frac{c_0}{r^2} & R_0 < r \le \frac{M}{2} \\
\frac{64c_0}{M^6} (M - r)^4 & \frac{M}{2} < r \le M \\
0 & r > M
  \end{cases}, \label{psi} \\ 
  c_0 \defeq (p(0) + (1 + 16c_V) c_V + \mu(0, R_0])R^2_0, \nonumber
  \end{gather}
  so that
  \begin{equation} \label{dpsi}
  d\psi = \mu - \frac{2c_0}{r^3}  \mathbf{1}_{(R_0, M/2]} - \frac{256c_0}{M^6} (M - r)^3 \mathbf{1}_{(M/2, M]}.  
  \end{equation}
Note that the choice of the numerical constant $64$ in \eqref{psi} makes $\psi$ continuous at $r = M/2$.

We construct the phase $\varphi$ by analysis of a differential equation involving $\psi$.
\begin{lemma} \label{phase lemma}
  There exists $\varphi(\cdot \, ;h) : (0, \infty) \to [0, \sqrt{\psi(R_0)}M]$ such that for all $h \in (0,h_0]$, $\varphi$ and $\varphi'$ are locally absolutely continuous, $\supp \varphi'( \cdot \, ; h) \subseteq [0, M]$, and 
  \begin{equation*}
  (\varphi')^2(r) - h \varphi''(r) = \psi(r), \qquad r \in (0, \infty). 
  \end{equation*}
  \end{lemma}
  
\begin{proof}

We shall build $\varphi$ in several steps. We begin by solving the initial value problem,
\begin{equation} \label{psi ode}
y'(r) = f_h(y(r), r), \qquad r \in (0, \infty), \, y(M) = 0. 
\end{equation}
where $f_h(x,r) \defeq h^{-1}(x^2 - \psi(r))$ is defined on the rectangle $[0, \sqrt{\psi(R_0)}]_x \times (0, \infty)_r$. By \cite[Chapter 2, Theorem 1.3]{cole}, there exists a small open interval $I \subseteq (0, \infty)$ containing $M$, and a solution $y$ to \eqref{psi ode} which is absolutely continuous in $I$. In fact, this solution is unique on $I$. For if $y_1$, $y_2$ are two solutions to \eqref{psi ode}, then $\tilde{y} \defeq y_1 - y_2$ solves $\tilde{y}' = h^{-1}(y_1 + y_2) \tilde{y}$, $\tilde{y}(M) = 0$, and hence is identically zero on $I$.

We next show that the solution $y$ to \eqref{psi ode} obtained in the previous paragraph $y$ extends to all of $(0, \infty)$ and obeys
\begin{gather}
0 \le y(r) \le \sqrt{\psi(R_0)}, \qquad r \in (0, \infty), \label{y bds}\\
y(r) = 0, \qquad r \ge M. \label{y vanishes past M}
\end{gather}
 This will allow us to conclude the construction of $\varphi$ by setting
\begin{equation} \label{from y to phi}
\varphi(r) \defeq \int_0^r y(s) ds.
\end{equation}

Let us first establish \eqref{y vanishes past M}. Because $y(M) = 0$, there exists $\epsilon \in (0,h)$ so that $[M , M + \e) \subseteq I$ and $|y(r) |\le 1/2$ on  $[M, M + \e)$. Therefore, using \eqref{psi ode} and \eqref{psi}, we see that $|y'(r)| = h^{-1}|y(r)|^2 \le (4h)^{-1}$ on $[M, M + \e)$. Hence for $r \in [M, M + \e)$,
\begin{equation*}
\begin{split}
|y(r)| \le \int^r_{M} |y'(s)| ds 
&\le \frac{\e}{4h} \le \frac{1}{4}.
\end{split}
\end{equation*}
Applying $|y'(r)| = h^{-1}|y(r)|^2$ on $[M, M + \e)$ another time, we then get $|y'(r)| \le (16h)^{-1}$ and use it to show that $|y(r)| \le 16^{-1}$, $ r \in [M  , M + \e)$. Continuing in this fashion, we see that $y(r) = 0$ for $r \in [M  , M + \e)$. Therefore $y$ extends to be identically zero on $[M, \infty)$. 

Moving on, we now confirm \eqref{y bds}. To see that $y \ge 0$, assume for contradiction that there exists $0< r_0 < M$ with $y(r_0) < 0$. Then, because $y' = h^{-1}(y^2 - \psi) \le h^{-1}y^2$,
\begin{equation} \label{bigger than zero}
\begin{split}
y(r_0)^{-1} - y(r)^{-1} & = \int^r_{r_0} \frac{y'(s)}{(y(s))^2} ds  \\
& \le \frac{r - r_0}{h}, \qquad \text{$r > r_0$, $r$ near $r_0$}.
\end{split}
\end{equation}
As $r$ approaches $\inf\{r \in [r_0, \infty): y(r) = 0 \} \le M$, \eqref{bigger than zero} must hold. But this is a contradiction because the left side becomes arbitrarily large, while the right side remains bounded. So $y(r) \ge 0$ where it is defined on $(0, M]$. 

To show $y \le \sqrt{\psi(R_0)}$, we compare $y$ to the solution of the initial value problem
\begin{equation*}
z' = (z^2 - \psi(R_0))/h, \qquad z(M) = 0.  \\
\end{equation*}
This solution exists for all $r > 0$ and is given by
\begin{equation*}
z(r) = \sqrt{\psi(R_0)} \tanh \big(h^{-1} \sqrt{\psi(R_0)}(M - r)\big).
\end{equation*}
where $\tanh$ denotes the hyperbolic tangent.
Suppose for contradiction that there exists $r_0 < M$ such that $y(r_0) > z(r_0)$. Set $\zeta \defeq y -z$. Then
$\zeta'  \ge h^{-1}(y+z)\zeta$, $\zeta(r_0) >0$, and $\zeta(M) =0$. 

Put $r_1 \defeq \inf \{r \in (r_0, M]: \zeta(r) = 0\}$. We derive a contradiction from

\begin{equation*}
-\zeta(r_0) = \int_{r_0}^{r_1} \zeta'(r) dr \ge h^{-1} \int_{r_0}^{r_1} (y+z)\zeta dr
\end{equation*}
because $-\zeta(r_0) < 0$, while $\int_{r_0}^{r_1} (y+z)\zeta dr \ge 0$ by the definition of $r_1$. 

So we have that  $0 \le y \le z \le \sqrt{\psi(R_0)}$ where it is defined on $(0, M)$. It then follows by \cite[Chapter 2, Theorem 1.3]{cole} that $y$ extends to all of $(0, M)$, where it obeys the same bounds.\\
\end{proof}

\subsection{Proof of key lower bound} \label{key lwr bd section}

\begin{lemma} \label{goal est lemma} 
Let $a$, $s$, and $\gamma$ be as in \eqref{V long range}, \eqref{s} and $\eqref{gamma}$, respectively. Let $w$ be as in \eqref{w} and $\varphi$ as constructed in Lemma \ref{phase lemma}. There exist $M \ge 2R_0 > 2$, $\kappa_1 \ge 0$, and $\kappa_2 \ge 1$ as in \eqref{w}, along with $C > 0$, all independent of $h$ and $\ep$, so that \eqref{w w prime est} holds and 
\begin{equation} \label{goal est}
\begin{split}
d(w((\varphi')^2 &- h\varphi'' )) -V_L w' - c_V(r+1)^{-1}m w -  w \mu - \tfrac{16 c_V^2 w}{ (r + 1)^{1+\delta_0}} -\tfrac{h^2 q}{4r^2} \\
 &\ge  ar \mathbf{1}_{(0,M]} + \langle r \rangle^{-2s}\mathbf{1}_{(M, \infty)}.
\end{split}
\end{equation}
\end{lemma}

\begin{proof}

First we show \eqref{goal est}. We have,
\begin{equation*}
\begin{split}
 d(w((\varphi')^2 &- h\varphi'' )) -V_L w' - c_V(r+1)^{-1}m w -  w \mu - \tfrac{16 c_V^2 w}{ (r + 1)^{1+\delta_0}} -\tfrac{h^2 q}{4r^2}  \\
&=(a + \psi - (V_L+a))w'  + w(d\psi - c_V(r+1)^{-1}m - \mu) \\
& - \tfrac{16 c^2_V w}{ (r + 1)^{1+\delta_0}} -\tfrac{h^2 q}{4r^2}.
\end{split}
\end{equation*}
Now estimate, using \eqref{R0} and \eqref{psi},
\begin{equation*}
\begin{split}
(a &+ \psi - (V_L + a))w' - \tfrac{16 c^2_V w}{ (r + 1)^{1+\delta_0}} -\tfrac{h^2 q}{4r^2} \\
&\ge w'(a + (p(0) + (1 + 16c_V)c_V + \mu(0, r]) \mathbf{1}_{(0, R_0]} + \tfrac{c_0}{r^2} \mathbf{1}_{(R_0, M/2]} + \tfrac{64c_0}{M^6} (M - r)^4 \mathbf{1}_{(M/2, M]} \\
&- p(0) \mathbf{1}_{(0, R_0]} - \tfrac{a}{4} \mathbf{1}_{(R_0, \infty)} - \tfrac{16 c^2_V w}{(r + 1)^{1+\delta_0}w'} ) -\tfrac{h^2 q}{4r^2}
\end{split}
\end{equation*}
From \eqref{s}, \eqref{R0}, and \eqref{w over dw}, 
\begin{equation*}
    \tfrac{16 c^2_V w}{(r + 1)^{1+\delta_0}w'} \le  16c_V^2 \mathbf{1}_{(0, R_0]} + \tfrac{a}{4} \mathbf{1}_{[R_0, \infty)}.
\end{equation*}
Therefore
\begin{equation*}
\begin{split}
(a &+ \psi - (V_L + a))w' - \tfrac{16 c^2_V w}{ (r + 1)^{1+\delta_0}} -\tfrac{h^2 q}{4r^2} \\
& \ge w'(\tfrac{3 a}{4} + c_V \mathbf{1}_{(0, R_0]} + \tfrac{c_0}{r^2} \mathbf{1}_{(R_0, M/2]} + \tfrac{64c_0}{M^6}(M - r)^4 \mathbf{1}_{(M/2, M]}) -\tfrac{h^2 q}{4r^2} \\
& = \tfrac{3 a}{2} r \mathbf{1}_{(0, M]} + \tfrac{3 a}{4} \max( \kappa_1 (r+1)^{-1} m, \kappa_2 \langle r \rangle^{-2s}) w \mathbf{1}_{(M, \infty)}\\
&+2c_V r \mathbf{1}_{(0, R_0]} +  \tfrac{2c_0}{r} \mathbf{1}_{(R_0, M/2]} + \tfrac{128c_0}{M^6}r(M - r)^4 \mathbf{1}_{(M/2, M]} -\tfrac{h^2 q}{4r^2},
\end{split}
\end{equation*}
where we used \eqref{w prime}.

On the other hand, by \eqref{V prime long range}, \eqref{R0}, \eqref{w}, and \eqref{dpsi},
\begin{equation*}
\begin{split}
w(d\psi - c_V(r + 1)^{-1}m - \mu) &=  w(\mu - \tfrac{2c_0}{r^3}  \mathbf{1}_{(R_0, M/2]} - \tfrac{256c_0}{M^6} (M - r)^3 \mathbf{1}_{(M/2, M]} \\
&- c_V(r+1)^{-1} m  \mathbf{1}_{(0,R_0] \cup (M, \infty)} - \tfrac{a}{4r} \mathbf{1}_{(R_0, M]} - \mu) \\
&\ge -\tfrac{2c_0}{r}  \mathbf{1}_{(R_0, M/2]} - \tfrac{256c_0}{M^6} r^2 (M - r)^3 \mathbf{1}_{(M/2, M]} \\
&- c_Vr^2 (r+1)^{-1} m \mathbf{1}_{(0,R_0]} - \tfrac{a}{4} r \mathbf{1}_{(R_0, M]} - c_V  (r + 1)^{-1}m w \mathbf{1}_{(M, \infty)}. 
\end{split}
\end{equation*}

Adding the previous two estimates,
\begin{equation} \label{pre pre final positivity}
\begin{split}
( \psi &- V_L)w'  + w(d\psi - c_V(r+ 1)^{-1}m - \mu) - \tfrac{16 c^2_V w}{ (r + 1)^{1+\delta_0}} -\tfrac{h^2 q}{4r^2} \\
& \ge r(\tfrac{3 a}{2}  \mathbf{1}_{(0,M]} - \tfrac{ a}{4} \mathbf{1}_{(R_0, M]} - \tfrac{128c_0}{M^6}(M-r)^3(3r - M) \mathbf{1}_{(M/2, M]})   \\
&+\tfrac{\kappa_2 a}{2} w \langle r \rangle^{-2s} \mathbf{1}_{(M, \infty)} + (\tfrac{\kappa_1 a}{4}- c_V) (r +1)^{-1}m(r) w \mathbf{1}_{(M, \infty)}-\tfrac{h^2 q}{4r^2}.
\end{split}
\end{equation}
First, on $[M/2, M]$, the maximum value of $(M-r)^3(3r-M)$ is $M^4/16$ at $r = M/2$. In view of this, choose $M$ large enough (depending on $a$ and $c_0$) so that $a/4 - 8c_0M^{-2} \ge 0$ ($M = (32c_0)^{1/2}a^{-1/2}$ suffices). Second, observe that by the definition of $q$ (see \eqref{w w prime est}) and \eqref{w}, $q(r) = 0$ for $r \in (0,M]$, while for $r > M$,
\begin{equation*}
-4^{-1}h^2 qr^{-2} = -4^{-1}h^2(2wr^{-1} - w')r^{-2} \ge - 2^{-1}h^2_0wr^{-3} \ge  - 2^{-1}h^2_0w(\langle r \rangle^{2s} r^{-3})\langle r \rangle^{-2s}. 
\end{equation*}
Applying all of these to \eqref{pre pre final positivity} gives
\begin{equation} \label{pre final positivity}
\begin{split}
( \psi &- V_L)w'  + w(d\psi - c_V(r + 1)^{-1}m - \mu) - \tfrac{16 c^2_V w}{ (r + 1)^{1+\delta_0}} -\tfrac{h^2 q}{4r^2} \\
& \ge  a r  \mathbf{1}_{(0,M]} + \tfrac{\kappa_2 a}{4} \langle r \rangle^{-2s} \mathbf{1}_{(M, \infty)} \\
&+ (\tfrac{\kappa_2 a}{4}  - \tfrac{h^2_0 }{2})  \langle r \rangle^{-2s} w \mathbf{1}_{(M, \infty)}  + (\tfrac{\kappa_1 a}{4}- c_V) (r +1)^{-1}m(r) w \mathbf{1}_{(M, \infty)}. 
\end{split}
\end{equation}

At this point, we fix $\kappa_1 \ge 0$ and $\kappa_2 \ge 1$ large enough so that 
\begin{equation*}
\tfrac{\kappa_1 a}{4}- c_V \ge 0, \qquad \tfrac{\kappa_2 a}{4}  - \tfrac{h^2_0 }{2} \ge 1,
\end{equation*}
completing the proof of \eqref{goal est}.

Finally, we show $q \ge 0$ for $r > M$. By \eqref{w} and \eqref{w prime}
\begin{equation} \label{W ge r over 2}
\begin{split}
2wr^{-1} - w' &= (2r\mathbf{1}_{(0, M)} +  2r^{-1} w \mathbf{1}_{(M, \infty)} ) \\
 &- (2r\mathbf{1}_{(0, M)} + \max (\kappa_1(r+1)^{-1} m(r),\, \kappa_2 \langle r \rangle^{-2s}) w\mathbf{1}_{(M, \infty)}) \\
 &= (2r^{-1} - \max (\kappa_1(r+1)^{-1} m(r), \, \kappa_2 \langle r \rangle^{-2s})) w \mathbf{1}_{(M, \infty)}.
\end{split}
\end{equation}
Now it may be necessary to increase $M$ so that $\kappa_1 (r + 1)^{-1} m(r), \, \kappa_2 \langle r \rangle^{-2s} \le 2r^{-1}$ for $r > M$ (recall $\lim_{r \to \infty} m(r) = 0$). \\
\end{proof}
\begin{remark} \label{optimal ext weight remark}

We reexamine the left side of \eqref{goal est} in the special case where $n \ge 3$ and, for all $h_0 \in (0, h_0]$ and $\ep \in [-\ep_0, \ep_0]$, $V_S = W_S = 0$ and $V_L(\cdot \, ; \ep, h)$ has support in $\overline{B(0,R_0)}$. If $n \ge 3$, we may drop the term $-h^2q(4r^2)^{-1}$ as explained before. Because the short range potentials vanish, we may disregard $-16c^2_Vw(r +1)^{-1-\delta_0}$ too. Finally, the support property of $V_L$ means we can ignore $-c_V(r + 1)^{-1} dr$. Given these simplifications we may take $\kappa_1 = 0$ and $\kappa_2 = 1$ (so \eqref{w w prime est} holds trivially because $2r^{-1} - \langle r \rangle^{-2s} \ge 0$), and we arrive at a streamlined version of \eqref{pre final positivity}:
\begin{equation} 
( \psi - V_L)w'  + w(d\psi - \mu)  \ge a r \mathbf{1}_{(0,M]} + \tfrac{ a}{2} \langle r \rangle^{-2s} w\mathbf{1}_{(M, \infty)}, 
\end{equation}
valid for $M = \max(2R_0, (32c_0)^{1/2}a^{-1/2})$ or larger.
\end{remark}

\subsection{Carleman estimate} \label{Carleman est subsection}
In this subsection, we prove the following semiclassical estimates, which are derived from a Carleman estimate established in the process.
 
\begin{lemma} \label{Carleman lemma}
Suppose that, for all $h \in (0, h_0]$ and $\ep \in [-\ep_0, \ep_0]$, $V(\cdot\, ; \ep, h)$ obeys \eqref{real part} through \eqref{W short range}. Let $s > 1/2$ be as in \eqref{s}. Let the weight $w$ and phase $\varphi$ be as designed in \eqref{w} and Lemma \ref{phase lemma}, respectively, with the constants $\gamma$, $R_0$, $M$, $\kappa_1$ and $\kappa_2$ as chosen over the course of Subsections \ref{phase and weight section} and \ref{key lwr bd section}. There exists $C > 0$ independent of $h$ and $\ep$ so that for all $h \in (0,h_0]$, $\ep \in [-\ep_0, \ep_0]$, and $v \in C_{0}^\infty(\mathbb{R}^n)$,
\begin{gather} 
\|\langle x \rangle^{-s} v \|^2_{L^2} \leq
  e^{C/h} \big( \|\langle x \rangle^{s} P( \ep,h)v \|^2_{L^2} 
+  \| |W_L|^{1/2} v \|^2_{L^2} \big), \label{Carleman est} \\
\|\langle x \rangle^{-s} \mathbf{1}_{> M} v \|^2_{L^2} \leq
  \frac{C}{h^2} \|\langle x \rangle^{s} P( \ep,h)v \|^2_{L^2} 
+   \frac{C}{h} \| |W_L|^{1/2} v \|^2_{L^2(\R^n)}. \label{ext Carleman est}
\end{gather}
\end{lemma}

The proof of Lemma \ref{Carleman lemma} proceeds in three steps. The first is to establish the away-from-origin Carleman estimate \eqref{just after integrate}, which has a loss at the origin, but immediately implies \eqref{ext Carleman est}. The second step is to use a modification of Obovu's result \cite[Lemma 2.2]{ob24}, which is based on Mellin transform techniques, to obtain an estimate for small $r$ which does not have a loss as $r \to 0$. In fact, the pertinent weight in Obovu's estimate is unbounded as $r \to 0$. We call this the near-origin estimate. The third and final step is to glue together the near-origin and away-from-origin estimates, to obtain \eqref{Carleman est}.

\begin{proof}[Proof of Lemma \ref{Carleman lemma}]

We remind the reader that we use the notation. $\|u\| \defeq \|u(r, \cdot)\|_{L^2(\US_\theta^{n-1})}$. In the following, $\int_{r,\theta}$ denotes the integral over $(0,\infty) \times \US^{n-1}$ with respect to the measure $dr d\theta$. Throughout, $C$ denotes a positive constant whose precise value changes, but is always independent of $h$ and $\ep$.

\subsubsection{Away-from-origin estimate}

We begin by combining \eqref{deriv wF 4} with \eqref{goal est}, which implies 
\begin{equation} \label{deriv wF 5}
\begin{split}
    d(wF) &\ge -\big(\tfrac{4w^2}{h^2w'} + \tfrac{c^2_V w^2}{2} \big)\|P_\varphi(\ep, h) u\|^2 \\
    &-d( \real ( c_2 w^2  \langle hu', hu \rangle)) \\
    &- h^{-1} \| |W_L|^{1/2} u\|^2  \big( \tfrac{5}{2} + \tfrac{ h^2 c_V w}{4r^2} + c_V w\|\real V\|_{L^\infty} \\
    &+ (4c^2_V + c_Vw)(\varphi')^2 + h c_V w |\varphi''| + 4c^2_V h^2 (dw)^2\big) \\
    &+ C^{-1} (\|u\|^2 + \|hu'\|^2) ( r \mathbf{1}_{(0,M]} +  \langle r \rangle^{-2s} \mathbf{1}_{(M, \infty)}). 
    \end{split}
\end{equation}
Recall that $u \in e^{\varphi/h} r^{(n-1)/2} C^\infty_0(\R^n)$ as in \eqref{function space u}.

Use \eqref{ftc} to integrate both sides of \eqref{deriv wF 5} over $(1/k, k]$, $k \in \NN$, with respect to $dr$. Then send $k \to \infty$. We have $wF(0) = 0$ since $w(r) = r^2$ near $r = 0$, while $wF(r) = 0$ for $r$ large since $u$ has compact support. The boundary terms coming from line two of \eqref{deriv wF 5} vanish too.  Therefore, for all $h \in (0, h_0]$ and $\ep \in [-\ep_0, \ep_0]$,
\begin{equation} \label{just after integrate}
\begin{split}
    \int_{r,\theta} & (|u|^2 + |hu'|^2)(r \mathbf{1}_{(0,M]} + \langle r \rangle^{-2s} \mathbf{1}_{(M, \infty)}) \\
      & \le \frac{C}{h^2} \int_{r,\theta} \langle r \rangle^{2s}|P_\varphi(\ep,h)u|^2  +  \frac{C}{h} \int_{r, \theta} |W_L| |u|^2.
 \end{split}
    \end{equation}
Here, we used that, $w^2/w' \le C \langle r \rangle^{2s}$. This is the away-from-origin estimate. Applying \eqref{conjugation} and that $\varphi(r) = \max \varphi$ for $r > M$, we divide both sides of \eqref{just after integrate} by $e^{2\max \varphi/h}$ to obtain \eqref{ext Carleman est}.
  
  \subsubsection{Near origin estimate}
  
  \begin{lemma} 
Fix $t_0 \in (-1/2, 0)$. There exist $C > 0$ independent of $\ep$ and $h$ so that for each $h \in (0, h_0]$, $\ep \in [-\ep_0, \ep_0]$, and $v \in C^\infty_0(\R^n)$,
\begin{equation} \label{obovu est}
\begin{split}
\int_{0 < r < 1/2, \theta} |r^{-\frac{1}{2} - t_0} r^{\frac{n-1}{2}} v |^2 &\le Ch^{-4} \big( \int_{0 < r < 1, \theta} |r^{\frac{3}{2} - t_0} r^{\frac{n-1}{2}} P(\ep, h) v |^2 \\
&+ \int_{\alpha < r < 1, \theta} |r^{\frac{3}{2} - t_0} V(\ep,h) r^{\frac{n-1}{2}}v |^2 \\
&+ h^4  \int_{1/2 < r < 1, \theta} |r^{\frac{3}{2} - t_0} r^{\frac{n-1}{2}} v  |^2 + h^2  \int_{1/2 < r < 1, \theta} |r^{\frac{3}{2} - t_0} h(r^{\frac{n-1}{2}} v)'   |^2 \big),
\end{split}
\end{equation}
where 
\begin{equation} \label{alpha}
\alpha \defeq \eta h,
\end{equation}
for some $\eta > 0$ independent of $h$ and $\ep$.
\end{lemma}
 
 \begin{proof}
 
The proof is only a small variation of the proof of \cite[Lemma 2.2]{ob24}, to allow for the potential to be complex valued, and to depend on $h$ and $\ep$. 
 
 Let $\chi \in C_0^\infty([0,\infty); \R)$ be such that $\chi = 1$ near $[0,1/2]$ and $\chi = 0$ near $[1, \infty)$. By \cite[(2.12)]{ob24}, for all $h > 0$,
 \begin{equation} \label{invoke obovu}
 \begin{split}
 \int_{r, \theta} | \chi r^{-\frac{1}{2} - t_0} r^{\frac{n-1}{2}} v |^2 & \le Ch^{-4} \big( \int_{0 < r < 1, \theta} |r^{\frac{3}{2} - t_0} r^{\frac{n-1}{2}} P(\ep, h) v |^2 \\
&+ \int_{0 < r < 1, \theta} |r^{\frac{3}{2} - t_0} V(\ep,h) r^{\frac{n-1}{2}}v |^2 \\
& + \int_{r, \theta} |r^{\frac{3}{2} - t_0} [r^{\frac{n-1}{2}} P(\ep, h) r^{-\frac{n-1}{2}}, \chi] r^{\frac{n-1}{2}} v |^2 \big).
 \end{split}
 \end{equation}
 Because the commutator reduces to
 \begin{equation*}
  [r^{\frac{n-1}{2}} P(\ep, h) r^{-\frac{n-1}{2}}, \chi] r^{\frac{n-1}{2}} v = -h^2(\chi'' r^{\frac{n-1}{2}} v + 2 \chi' (r^{\frac{n-1}{2}} v)'),
 \end{equation*}
 \eqref{invoke obovu} implies
  \begin{equation} \label{invoke obovu consequence}
 \begin{split}
 \int_{r, \theta} | \chi r^{-\frac{1}{2} - t_0} r^{\frac{n-1}{2}} v |^2 & \le Ch^{-4} \big( \int_{0 < r < 1, \theta} |r^{\frac{3}{2} - t_0} r^{\frac{n-1}{2}} P(\ep, h) v |^2 \\
&+ \int_{0 < r < 1, \theta} |r^{\frac{3}{2} - t_0} V(\ep,h) r^{\frac{n-1}{2}}v |^2 \\ 
&+ h^4  \int_{1/2 < r < 1, \theta} |r^{\frac{3}{2} - t_0} r^{\frac{n-1}{2}} v  |^2 + h^2  \int_{1/2 < r < 1, \theta} |r^{\frac{3}{2} - t_0} h(r^{\frac{n-1}{2}} v)'   |^2 \big).
 \end{split}
 \end{equation}
 Now, considering the term in line two of \eqref{invoke obovu consequence}, we decompose integration in $r$ with respect to $\alpha = \eta h$. Supposing $\eta < h^{-1}_0$ so that $\alpha < 1$,
 \begin{equation} \label{decompose using alpha}
 \begin{split}
Ch^{-4} \int_{0 < r < 1, \theta}& |r^{\frac{3}{2} - t_0} V(\ep,h) r^{\frac{n-1}{2}}v |^2 \\
&= Ch^{-4} \int_{0 < r < \alpha, \theta} |r^{\frac{3}{2} - t_0} V(\ep,h) r^{\frac{n-1}{2}}v |^2 + Ch^{-4}  \int_{\alpha < r < 1, \theta} |r^{\frac{3}{2} - t_0} V(\ep,h) r^{\frac{n-1}{2}}v |^2 \\
 &=Ch^{-4} \alpha^4 \int_{0 < r < \alpha, \theta} |r^{-\frac{1}{2} - t_0} r^{\frac{n-1}{2}}v |^2 + Ch^{-4} \int_{\alpha < r < 1, \theta} |r^{\frac{3}{2} - t_0} V(\ep,h) r^{\frac{n-1}{2}}v |^2 \\
 & =C \eta^{4} \int_{0 < r < \alpha, \theta} |r^{-\frac{1}{2} - t_0} r^{\frac{n-1}{2}}v |^2 + Ch^{-4} \int_{\alpha < r < 1, \theta} |r^{\frac{3}{2} - t_0} V(\ep,h) r^{\frac{n-1}{2}}v |^2,
 \end{split} 
 \end{equation}
 Taking $\eta$ smaller if necessary, depending on $C$ but independent of $h$ and $\ep$, we can absorb the first term in line four of \eqref{decompose using alpha} into the left side of \eqref{invoke obovu consequence}, completing the proof of \eqref{obovu est}.\\
 \end{proof}
 
 \subsubsection{Combining the near-origin and away-from-origin estimates}
 
 For $v \in C^\infty_0(\R^n)$, set $\tilde{u} = r^{\frac{n-1}{2}} v$. We have
\begin{equation} \label{begin gluing strategy}
\begin{split}
\| \langle r \rangle^{-s} v \|^2_{L^2} &=  \int_{0 < r < 1/2, \theta} |\langle r \rangle^{-s} \tilde{u}|^2 + \int_{r > 1/2, \theta} |\langle r \rangle^{-s} \tilde{u}|^2 \\
&\le C \int_{0 < r < 1/2, \theta} |r^{-\frac{1}{2} - t_0} \tilde{u}|^2+ C \int_{r > 1/2, \theta}( r \mathbf{1}_{(0,M]} +  \langle r \rangle^{-2s} \mathbf{1}_{(M, \infty)})|\tilde{u}|^2, 
\end{split}
\end{equation}
where we used $\langle r \rangle^{-2s} \le C(r \mathbf{1}_{(0,M]} +  \langle r \rangle^{-2s} \mathbf{1}_{(M, \infty)})$ for $r > 1/2$. Let us bound the first term of the second line of \eqref{begin gluing strategy} using \eqref{obovu est}. In doing so we apply
\begin{gather}
 \int_{0 < r < 1, \theta} |r^{\frac{3}{2} - t_0} r^{\frac{n-1}{2}} P(\ep,h) v |^2  \le C \int_{r , \theta} |\langle r \rangle^{s} P_\varphi(\ep, h) u |^2, \label{handle P straightforwardly} \\
 \int_{\alpha < r < 1, \theta} |r^{\frac{3}{2} - t_0} V(\ep, h) \tilde{u} |^2 \le C \int_{r , \theta} (r \mathbf{1}_{(0,M]} +  \langle r \rangle^{-2s} \mathbf{1}_{(M, \infty)}) |u|^2, \label{convert V} \\
  \int_{1/2< r < 1, \theta} |r^{\frac{3}{2} - t_0} \tilde{u} |^2 +   \int_{1/2 < r < 1, \theta} |r^{\frac{3}{2} - t_0} h\tilde{u}' |^2 \le C \int_{r, \theta} (r \mathbf{1}_{(0,M]} +  \langle r \rangle^{-2s} \mathbf{1}_{(M, \infty)})( |u |^2 + |h u '|^2), \label{est commutator terms}
 \end{gather}
where, as in \eqref{function space u}, $u = e^{\varphi/h} r^{(n-1)/2} v = e^{\varphi/h} \tilde{u}$. To get \eqref{est commutator terms}, we used, for $1/2 < r < 1$,
\begin{equation*}
| r^{\frac{1}{2} - t_0} h\tilde{u}'|^2 = |r^{\frac{1}{2} - t_0}(e^{-\frac{\varphi}{h}} hu' -  \varphi'e^{-\frac{\varphi}{h}} u)|^2 \le C(r \mathbf{1}_{(0,M]} +  \langle r \rangle^{-2s} \mathbf{1}_{(M, \infty)})( |u |^2 + |h u '|^2).
\end{equation*}
The upshot is that \eqref{begin gluing strategy} implies
\begin{equation} \label{last step in gluing}
\| \langle r \rangle^{-s} v \|^2_{L^2} \le Ch^{-4}\big( \int_{r , \theta} |\langle r \rangle^{s} P_\varphi(\ep, h) u |^2 + \int_{r , \theta}( r \mathbf{1}_{(0,M]} +  \langle r \rangle^{-2s} \mathbf{1}_{(M, \infty)})( |u |^2 + |h u '|^2) \big).
\end{equation}

The proof of \eqref{Carleman est} is then completed by using \eqref{just after integrate} (the away-from-origin estimate) to bound the second term on the right side of \eqref{last step in gluing}.\\
  \end{proof}
  
Combining \eqref{Carleman est} and \eqref{ext Carleman est},
\begin{equation} \label{add ests together}
\begin{split}
e^{-C/h}&  \| \langle x \rangle^{-s} \mathbf{1}_{\{|x| \le M\}} v \|^2_{L^2} +  \| \langle x \rangle^{-s} \mathbf{1}_{\{|x| > M\}} v \|^2_{L^2} \\
&\le \frac{C}{h^2} \| \langle x \rangle^{-s} P(\ep, h) v \|^2_{L^2}  + \frac{C}{h} \| |W_L|^{1/2} v \|^2_{L^2(\R^n)}.
\end{split}
\end{equation}

Recall that in subsection \ref{regularity and decay of potential subsection} we supposed $\pm W_L, \pm W_S \ge 0$.  Using this, we estimate the second term in the second line of \eqref{add ests together}. Our convention for the $L^2$-inner product is that complex conjugation takes place in the first argument; for all $\gamma_0, \gamma_1 > 0$,
\begin{equation} \label{use sign of imaginary part}
\begin{split}
 \| |W_L|^{1/2} v \|^2_{L^2} &= -  \imag \langle \pm iW_L v, v \rangle_{L^2}  \\
 &\le - \imag \langle   \pm i W v, v \rangle_{L^2} \\
 &= \mp  \imag \langle P(\ep, h) v, v \rangle_{L^2} \\
 & \le  \tfrac{\gamma_0^{-1}}{2}  \| \langle x \rangle^s \mathbf{1}_{\{|x| \le M\}} P(\ep, h) v \|^2_{L^2} + \tfrac{ \gamma_0}{2}  \| \langle x \rangle^{-s} \mathbf{1}_{\{|x| \le M\}} v \|^2_{L^2} \\
 &+  \tfrac{\gamma^{-1}}{2}  \| \langle x \rangle^s \mathbf{1}_{\{|x| > M\}} P(\ep, h) v \|^2_{L^2(\R^n)} + \tfrac{\gamma}{2} \| \langle x \rangle^{-s} \mathbf{1}_{\{|x| > M\}} v \|^2_{L^2}.
 \end{split}
\end{equation}
Setting $\gamma_0 = C^{-1}he^{-C/h}$ and $\gamma_1 = C^{-1} h$, we absorb the terms involving $\langle x \rangle^{-s} \mathbf{1}_{\{|x| > M\}} v$ or $\langle x \rangle^{-s} \mathbf{1}_{\{|x| \le M\}} v$ on the right side of \eqref{use sign of imaginary part} into the left side of \eqref{add ests together}. We thus have
\begin{equation} \label{Carleman est no remainder}
\begin{split}
e^{-C/h} \|\langle x \rangle^{-s}& \mathbf{1}_{\{|x| \le M\}} v \|^2_{L^2} + \|\langle x \rangle^{-s} \mathbf{1}_{\{|x| > M\}} v \|^2_{L^2} \\
 &\le e^{C/h}  \| \langle x \rangle^s \mathbf{1}_{\{|x| \le M\}} P(\ep, h) v \|^2_{L^2}+ \frac{C}{h^2}  \| \langle x \rangle^s \mathbf{1}_{\{|x| > M\}} P(\ep, h) v \|^2_{L^2}.
\end{split}
\end{equation}

\appendix

\section{Proof of \eqref{use std density}}

\label{resolv est appendix}

For $s > 0$ the operator
\begin{equation*}
[P(\ep,h), \langle x \rangle^{s}]\langle x \rangle^{-s} = \left(-h^2( \Delta \langle x \rangle^{s}) - 2h^2 (\nabla \langle x \rangle^{s}) \cdot \nabla \right) \langle x \rangle^{-s}
\end{equation*}
is bounded $H^2(\R^n) \to L^2(\R^n)$. So, for $v \in H^2(\R^n)$ such that $\langle x \rangle^{s} v \in H^2(\R^n)$,
 \begin{equation}\label{weight times v in H2}
 \begin{split}
\|\langle x \rangle^{s}P(\ep, h)v\|_{L^2}  \le \|P(\ep, h))\langle x \rangle^{s} v \|_{L^2} +  \|[P(\ep, h),\langle x \rangle^{s}]\langle x \rangle^{-s}\langle x \rangle^{s}v \|_{L^2}
 \le C \| \langle x \rangle^{s}v \|_{H^2},
\end{split} 
\end{equation}
for some $C > 0$ independent of $v$.

Given $1/2 < s< 1$ and $f \in L^2(\R^n)$, the function $ u= \langle x \rangle^{s}(P(\ep, h))^{-1}\langle x \rangle^{-s} f$ belongs to $\langle x \rangle^s H^2(\R^n)$. This follows from Lemma \ref{apply cutoff lem} below because
\begin{equation*}
P(\ep, h) u =  f + [P(h), \langle x \rangle^{s}]  \langle x \rangle^{-s}  u \in L^2(\R^n),
\end{equation*}
with  $[P(h), \langle x \rangle^{s}] $ being bounded $H^2(\R^n) \to L^2(\R^n)$ since $s < 1$. 

Now, choose a sequence $v_k \in C_{0}^\infty$ such that $ v_k \to  \langle x \rangle^{s}(P(\ep, h))^{-1}\langle x \rangle^{-s} f$ in $H^2(\R^n)$. Define $\tilde{v}_k \defeq \langle x \rangle^{-s}v_k$. Then, as $k \to \infty$,
\begin{equation*}
\begin{split}
\| \langle x \rangle^{-s} \tilde{v}_k - \langle x \rangle^{-s} (P(\ep,h))^{-1}\langle x \rangle^{-s}f \|_{L^2}  
\le \| v_k - \langle x \rangle^{s} (P(\ep,h))^{-1}\langle x \rangle^{-s}f \|_{H^2} \to 0.
\end{split}
\end{equation*}
Also, applying equation \eqref{weight times v in H2},
\begin{equation*}
\|\langle x \rangle^{s}P(\ep, h)\tilde v_k - f\|_{L^2} \le C \|v_k - \langle x \rangle^{s} (P(\ep,h))^{-1} \langle x \rangle^{-s} f \|_{H^2} \to 0.
\end{equation*} 
Thus \eqref{use std density} follows by replacing $v$ by $\tilde{v}_k$ in \eqref{prepare to use density} and sending $k \to \infty$.

\begin{lemma} \label{apply cutoff lem}
If $u \in \langle x \rangle H^2(\R^n)$ and if $f \defeq P(\ep, h)u$, defined as a distribution, belongs to $L^2(\R^n)$, then in fact $u \in H^2(\R^n)$ and $u = (P(\ep, h))^{-1} f$.
\end{lemma} 

\begin{remark}
The proof shows that this lemma holds also if $P(\ep, h)$ is replaced by $-c^2 \Delta + V - \lambda^2$, for $\imag \lambda > 0$, $c$ obeying \eqref{c bd above and below}, and $V \in L^\infty(\R^n)$.  
\end{remark}

\begin{proof}
Let $\chi \in C^\infty_0(\R^n ; [0,1])$ be such that $\chi = 1$ near $B(0,1)$ with $\supp \chi \subseteq B(0,2)$. For $R > 0$, put $\chi_R(x) \defeq \chi(x/R)$. Then $\chi_R u \in H^2(\R^n)$ and 
\begin{equation*}
P(\ep, h) \chi_R u = f + [P(\ep, h), \chi_R] u = f -h^2( \Delta \chi_R) u -2 h^2 \nabla \chi_R \cdot \nabla u. 
\end{equation*}
We have $\nabla \chi_R = O(R^{-1})$ and $\Delta \chi_R = O(R^{-2})$, both of which have support in $\{R \le |x| \le 2R\}$. Therefore, because $u \in \langle x \rangle H^2(\R^n)$, in follows that $h^2( \Delta \chi_R) u +2 h^2 \nabla \chi_R \cdot \nabla u$ converges to zero in $L^2(\R^n)$ as $R \to \infty$. So in the sense of $L^2$-convergence
\begin{equation*}
u = \lim_{R \to \infty} \chi_R u = (P(\ep,h))^{-1} f. 
\end{equation*}
\end{proof}

\section{Free resolvent at low frequency}
\label{free resolv low freq appendix}

In this appendix, we deduce H\"older regularity for
\begin{equation} \label{find low freq exp for} \langle \cdot \rangle^{-s} \big( (-\Delta -\lambda^2)^{-1} +\tfrac{1}{2\pi} \log \big( \tfrac{-i\lambda|x-y|}{2} \big) \big) \langle \cdot \rangle^{-s}: L^2(\R^n)\to L^2(\R^n),  \qquad n \ge 2, \, s > 1,
\end{equation}
for $\lambda$ in compact subsets of $\imag \lambda \ge 0$. The logarithmic term in \eqref{find low freq exp for} should be omitted except when $n = 2$. We employ the notation $R_0(\lambda) \defeq (-\Delta - \lambda^2)^{-1}$.  

To begin, recall the well known formula for the integral kernel of the free resolvent \cite[(3.1)]{jene01}, 
\begin{equation} \label{resolv and Hankel}
R_0(\lambda)(|x-y|) = \frac{i}{4} \Big( \frac{\lambda}{2\pi |x-y|} \Big)^{\frac{n}{2} -1} H^{(1)}_{\frac{n}{2} -1} (\lambda |x-y|), \qquad \imag \lambda > 0, 
\end{equation}
where $H^{(1)}_\nu$ is principal branch of the Hankel function of the first kind of order $\nu$ \cite[\S 10.2(ii)]{dlmf}.


Next, we use the relationship between $H^{(1)}_\nu$ and the Macdonald function $K_\nu$ \cite[10.27.4, 10.27.5, 10.27.8]{dlmf}. Setting $\nu \defeq (n/2) - 1$,
\begin{equation*}
H^{(1)}_{\nu} (\lambda |x-y|) = H^{(1)}_{\nu} (i(-i\lambda |x-y|)) = \frac{2}{i \pi} e^{- i\pi \nu/2} K_{\nu}(-i \lambda |x-y|).
\end{equation*}

Combining this with \eqref{resolv and Hankel} yields
\begin{equation} \label{resolv and macdonald}
R_0(\lambda)(|x-y|) = \frac{1}{2\pi}\Big( \frac{-i\lambda}{2\pi|x-y|} \Big)^{\nu} K_{\nu}(-i \lambda |x-y|), \quad \nu = \frac{n}{2} -1.
\end{equation}


First we concentrate on $n = 2$, where $\nu = 0$. Put 
\begin{equation} \label{A0}
A(z) \defeq K_0(z) + \log \big( \frac{z}{2} \big),
\end{equation}
so that 
\begin{equation} \label{free resolv exp dim two}
\langle x \rangle^{-s} R_0(\lambda)(|x-y|)\langle y \rangle^{-s} = \frac{1}{2\pi} \langle x \rangle^{-s} A(-i \lambda |x-y|) \langle y \rangle^{-s} - \frac{1}{2\pi} \langle x \rangle^{-s} \log \big( \frac{-i\lambda|x-y|}{2} \big)\langle y \rangle^{-s}.
\end{equation}
From \cite[10.31.2]{dlmf}, we see that $A(z)$ tends to Euler's constant $-\gamma$ for $\real z > 0$ and $z \to 0$. Using then the recurrence relation $\partial_z K_0 = -K_1$ \cite[10.29.3]{dlmf},
\begin{equation*}
\partial_z A(z) =  -K_1(z) + \frac{1}{z},
\end{equation*}
From \cite[10.30.2]{dlmf}, we see that $-K_1(z) + (1/z)$ goes to zero for $\real z > 0$ and $z \to 0$. Furthermore, for any $\nu$ \cite[10.25.3]{dlmf},
\begin{equation} \label{macdonald asymptotic infinity}
K_\nu(z) \sim (\pi/(2z))^{1/2} e^{-z}, \qquad z \to \infty. 
\end{equation}
Thus we conclude that, in $\real z > 0$, $A(z)$ is complex differentiable, $\partial_zA$ is bounded, and for any $\epsilon > 0$ there exists $C_\epsilon > 0$ so that 
\begin{equation*}
|A(z)| \le C_\epsilon(1 + |z|^{\epsilon}).
\end{equation*}
Hence, for $\lambda$ in the upper half plane,
\begin{equation*}
\begin{gathered}
|\langle x \rangle^{-s} A(-i\lambda |x -y|) \langle y \rangle^{-s}| \le C_\epsilon \langle x \rangle^{-s} (1 + |\lambda|^\epsilon|x -y|^{\epsilon}) \langle y \rangle^{-s}, \\
| \partial_\lambda \langle x \rangle^{-s} A(-i\lambda |x -y|) \langle y \rangle^{-s}| \le \sup_{\real z > 0} (| \partial_zA(z)|) \langle x \rangle^{-s} |x -y| \langle y \rangle^{-s}.
\end{gathered} 
\end{equation*}
For $s > 2$, the kernel $\langle x \rangle^{-s} |x-y| \langle y \rangle^{-s}$ is Hilbert-Schmidt. On the other hand for $s > 1$, the kernel  $\langle x \rangle^{-s} |x-y|^{\epsilon} \langle y \rangle^{-s}$ is Hilbert-Schmidt for $\epsilon > 0$ small enough. Therefore,
\begin{equation*}
\lambda \mapsto \langle x \rangle^{-s} A(-i \lambda |x-y|) \langle y \rangle^{-s}
\end{equation*}
is continuous from $\imag \lambda \ge 0$ to the space of bounded operators $L^2(\R^2) \to L^2(\R^2)$ for $s > 1$, and continuously differentiable if $s > 2$.

For $n \ge 3$, we use $\tfrac{d}{d\lambda} \langle \cdot \rangle^{-s} (-\Delta - \lambda^2)^{-1} \langle \cdot \rangle^{-s} = 2 \lambda \langle \cdot \rangle^{-s} (-\Delta - \lambda^2)^{-2} \langle \cdot \rangle^{-s}$ and two lemmas:

\begin{lemma}[{\cite[Proposition 2.4]{gimo74}}] \label{gimo74 lemma}
Let $n \ge 3$, $s_1, \, s_2 > 1/2$, and $s_1 + s_2 > 2$. Then $\langle \cdot \rangle^{-s_1}R_0(\lambda)  \langle \cdot \rangle^{-s_2} : L^2(\R^n) \to H^2(\R^n)$ extends continuously to $\imag \lambda \ge 0$. 
\end{lemma} 

\begin{remark}
By Corollary \eqref{cont ext cor} and Remark \eqref{reduce three halves}, if $s > 1/2$,  $\langle \cdot \rangle^{-s}R_0(\lambda)  \langle \cdot \rangle^{-s} : L^2(\R^n) \to L^2(\R^n)$ has a continuous extension to $(-\infty, \lambda_0] \cup [\lambda_0, \infty)$ for any $\lambda_0 > 0$. The additional restriction on the weights is necessary so that the extension may be taken to all of the $\R$. The proof of Lemma \eqref{gimo74 lemma} in \cite{gimo74} uses the Fourier transform to reduce the study of \eqref{resolv and macdonald} to the case $n = 3$. 
\end{remark}

\begin{lemma}[{\cite[Lemma 3.2]{llst24}}] \label{lap free resolv square lem}
Let $n \ge 3$ and 
\begin{equation} \label{s restrictions}
s > \begin{cases}
\frac{n + 3}{4}  & n \neq 8, \\
3 & n = 8. 
\end{cases}
\end{equation}
There exists $C > 0$ such that for all $\lambda \in \C$ with $\imag \lambda > 0$, 
\begin{equation} \label{lap free resolv square}
\| \lambda \langle x \rangle^{-s} (-\Delta - \lambda^2)^{-2} \langle x \rangle^{-s}\|_{L^2(\R^n) \to L^2(\R^n)} \le C(1 + |\lambda|)^{-1} 
\end{equation}
\end{lemma} 

\begin{remark}
The proof of Lemma \ref{lap free resolv square lem} in \cite{llst24} involves differentiating \eqref{resolv and macdonald} and checking for which $s$ the resulting weighted kernel is Hilbert-Schimdt or satisfies the hypotheses of the Schur test \cite[Section A.5]{dz}.
\end{remark}

We now establish H\"older continuity of \eqref{find low freq exp for} in compact subsets of $\imag \lambda \ge 0$ (recall that the logarithmic term in \eqref{find low freq exp for} is omitted except when $n = 2$). For this, fix $s > 1$, $A > 0$, and take $s_0, \, s_1$, such that $1 < s_0 < s < s_1$ and 
\begin{equation*}
s_1 > \begin{cases}
2 & n = 2, \\
\frac{n + 3}{4}  & n \ge 3, \, n \neq 8, \\
3 & n = 8. 
\end{cases}
\end{equation*}
Our work above when $n = 2$, as well as Lemmas \ref{gimo74 lemma} and \ref{lap free resolv square lem}, shows that there exist $C_j > 0$, $j \in \{0, 1\}$, so that for all $\lambda_1, \lambda_2$ with $\imag \lambda_1, \, \imag \lambda_2 > 0$ and $| \lambda_1|, |\lambda_2| \le A$,
\begin{equation*}
\begin{split}
\| \langle \cdot \rangle^{-s_j} \Big( R_0(\lambda_2) - R_0(\lambda_1) &+\tfrac{1}{2\pi} (\log \big( \tfrac{-i\lambda_2|x-y|}{2} \big) -\log \big( \tfrac{-i\lambda_1|x-y|}{2} \big) \big) \Big) \langle \cdot \rangle^{-s_j} \|_{L^2(\R^n) \to L^2(\R^n)} \\
& \le C_j |\lambda_2 - \lambda_1|^j
\end{split} 
\end{equation*} 

Now, with $\lambda_1$ and $\lambda_2$ fixed, consider the mapping
\begin{equation*}
\sigma \mapsto \langle \cdot \rangle^{-\sigma} \Big( R_0(\lambda_2) - R_0(\lambda_1) +\tfrac{1}{2\pi} (\log \big( \tfrac{-i\lambda_2|x-y|}{2} \big) -\log \big( \tfrac{-i\lambda_1|x-y|}{2} \big) \big) \Big) \langle \cdot \rangle^{-\sigma}
\end{equation*} 
which is holomorphic from $s_0 < \real \sigma < s_1$ to the space of bounded operators $L^2(\R^n) \to L^2(\R^n)$. Using the above bounds on the operator norm on the strips $\real \sigma = s_0$ and $\real \sigma = s_1$, the three lines lemma gives

\begin{equation*}
\begin{split}
\| \langle \cdot \rangle^{-s} \Big( R_0(\lambda_2) - R_0(\lambda_1) &+\tfrac{1}{2\pi} (\log \big( \tfrac{-i\lambda_2|x-y|}{2} \big) -\log \big( \tfrac{-i\lambda_1|x-y|}{2} \big) \big) \Big) \langle \cdot \rangle^{-s} \|_{L^2(\R^n) \to L^2(\R^n)} \\
& \le C^{1 -t}_0C^{t}_1 |\lambda_2 - \lambda_1|^t,
\end{split} 
\end{equation*}
where $t \in (0,1)$ is such that $(1 - t)s_0 + ts_1 = s$.

To upgrade to H\"older continuity $L^2(\R^2) \to H^2(\R^2)$, use Lemma \ref{simple Sobolev est lem} below, in combination with the identities
\begin{equation*}
\begin{gathered}
\langle \cdot \rangle^{-s} \Delta (-c^2\Delta - \lambda^2)^{-1} \langle \cdot \rangle^{-s} = -c^{-2}\langle \cdot \rangle^{-2s} - c^{-2} \lambda^2 \langle \cdot \rangle^{-s} (-c^2\Delta - \lambda^2)^{-1} \langle \cdot \rangle^{-s}, \\
 \Delta_{_x} \int_{\R^2} \log(-i\lambda|x-y|) \langle y \rangle^{-s} c^{-2}(y) f(y) dy = \langle x \rangle^{-s} c^{-2}(x) f(x) , \qquad x \in \R^2.
 \end{gathered}
\end{equation*}

\section{Resolvent with potential at low frequency}

\label{perturb resolv low freq appendix}

The proofs in this subsection are based on \cite[Section 2 and Appendix]{vo04} and \cite[Section 3]{llst24}. We consider $n \ge 3$ and $V \in L^\infty(\R^n; [0, \infty))$. We suppose there exist $C, \rho$ such that $|V(x)| \le C \langle x \rangle^{-\rho}$ where
\begin{equation} \label{srp decay appendix}
\rho > \begin{cases}
7/2 & \text{if } n = 3, \\
5 & \text{if } n = 4, \\
\max(3, n/2) & \text{if } n \ge 5.
\end{cases}
\end{equation}
When $n = 4$ we suppose in addition that the distributional derivatives $\partial_{x_j} V$ of $V$, $1 \le j \le 4$, belong to $L^\infty(\R^4)$. Our goal is to show that for any $s > 1$, the mapping
\begin{equation*}
\lambda \to   \langle x \rangle^{-s} (- \Delta + V - \lambda^2)^{-1} \langle x \rangle^{-s},\end{equation*}
with values in the space of bounded operators $L^2(\R^n) \to H^2(\R^n)$, is H\"older continuous for $\lambda$ in compact subsets of $\imag \lambda \ge 0$. Clearly, it suffices to the show this for $0 < s-1 \ll 1$. 

Our starting point is the resolvent identity 
\begin{equation} \label{resolv id for deriv}
(-\Delta + V - \lambda^2)^{-1} \langle x \rangle^{-s} (I + K(\lambda)) =  R_0(\lambda)\langle x \rangle^{-s}, 
\end{equation}
where $K(\lambda) \defeq V(x) \langle x \rangle^{s + s'} \langle x \rangle^{-s'} R_0(\lambda) \langle x \rangle^{-s}$, and as before we use $R_0(\lambda) \defeq (-\Delta - \lambda^2)^{-1}$. Here, $1/2 < s'  < s$ so that $s' + s > 2$. Then, by Lemma \ref{gimo74 lemma}, $R_{0,s',s}(\lambda) \defeq \langle x \rangle^{-s'} R_0(\lambda)\langle x \rangle^{-s} : L^2(\R^n) \to L^2(\R^n)$ extends continuously from $\imag \lambda > 0$ to $\R$.

The desired H\"older continuity follows if we show $I + K(\lambda)$ is invertible $L^2(\R^n) \to L^2(\R^n)$ for $\imag \lambda \ge 0$. For then
\begin{equation} \label{resolv id with K}
\langle x \rangle^{-s} (-\Delta + V - \lambda^2)^{-1} \langle x \rangle^{-s}  = \langle x \rangle^{-s}  R_0(\lambda)\langle x \rangle^{-s} (I + K(\lambda))^{-1}.
\end{equation}
By the previous appendix, \eqref{resolv id with K} exhibits $\langle x \rangle^{-s} (-\Delta + V - \lambda^2)^{-1} \langle x \rangle^{-s}$ as a product of two H\"older continuous mappings, since $(I + K(\lambda_2))^{-1} - (I + K(\lambda_1))^{-1} =  (I + K(\lambda_1))^{-1} (K(\lambda_1) - K(\lambda_2))(I + K(\lambda_2))^{-1}$.

It holds that $K(\lambda)$ is a compact $L^2(\R^n) \to L^2(\R^n)$, on account of \cite[Theorem B.4]{dz}. Hence, by the Fredholm alternative, $I + K(\lambda)$ is invertible if we can show $(I + K(\lambda))g = 0$ implies $g = 0$. To this end, put $u \defeq \langle x \rangle^{s'} R_{0, s',s}(\lambda)g$, which belongs to $\langle x \rangle^{s'} H^2(\R^n)$. If we can show $u = 0$, then in fact $g = 0$. This is because $(-\Delta - \lambda^2) u = \langle x \rangle^{-s} g$ in the distributional sense.

First, suppose $\lambda^2 \in \C \setminus [0, \infty)$. Then $u = 0$ follows immediately from $(-\Delta + V - \lambda^2)u  = \langle x \rangle^{-s}g + V R_0(\lambda) \langle x \rangle^{-s} g = \langle x \rangle^{-s}(1 + K(\lambda))g = 0.$ If $\lambda^2 \in (0, \infty)$, the idea is the same, but we incorporate a limiting step that uses \eqref{lap G} (which applies in this case since  \eqref{srp decay appendix} implies \eqref{srp V decay}). Set $u_{\ep} = (-\Delta - (\lambda + i\ep)^2)^{-1} \langle x \rangle^{-s}g$. Then $\langle x \rangle^{-s'} u_{\ep}$ converges to $\langle x \rangle^{-s'}u$ in $H^2(\R^n)$ as $\ep \to 0^+$. Moreover,
\begin{equation*}
\begin{split}
u_{\ep} &= (-\Delta + V - (\lambda + i\ep)^2)^{-1} (-\Delta + V - (\lambda + i\ep)^2) (-\Delta - (\lambda + i \ep)^2)^{-1} \langle x \rangle^{-s}g \\
&= (-\Delta + V - (\lambda + i\ep)^2)^{-1} \langle x \rangle^{-s} (I + V \langle x \rangle^{s}(-\Delta - (\lambda + i \ep)^2)^{-1} \langle x \rangle^{-s}) g
\end{split}
\end{equation*} 
Therefore, by \eqref{lap G}, for some $C > 0$ independent of $\ep$,
\begin{equation} \label{u ep}
\begin{split}
\| \langle x \rangle^{-s'}u \|_{L^2} &= \lim_{\ep \to 0^+} \| \langle x \rangle^{-s'} u_{\ep} \|_{L^2} \\
&\le C \lim_{\ep \to 0^+} \| (I + V \langle x \rangle^{s}(-\Delta - \lambda^2 \pm i \ep)^{-1} \langle x \rangle^{-s}) g \|_{L^2} \\
&= \| (I + K^\pm(\lambda)) g \|_{L^2} = 0.
\end{split} 
\end{equation}    

It remains to obtain invertibility of $I + K(\lambda)$ when $\lambda = 0$. For $n \ge 5$, it was shown in \cite[Section 5]{lsv25}, that if $V \in L^\infty(\R^n ; [0, \infty))$ and $V = O(\langle x \rangle^{-\rho})$ with $\rho > \max(3, n/2)$, there exist $C > 0$, $0 < \kappa \ll 1$ such that for $\imag \lambda > 0$ and $|\lambda| < \kappa$,
\begin{equation} \label{no zero resonance condition}
\| \langle x \rangle^{-s} (- \Delta + V - \lambda^2)^{-1} \langle x \rangle^{-s}\|_{L^2 \to L^2} \le C.
\end{equation}
Then an estimate similar to \eqref{u ep} establishes $u = 0$.

It remains to investigate the $\lambda = 0$ case when $n = 3$ or $n = 4$. We tackle these more directly. Indeed, 
\begin{equation} \label{use eqn lambda equals zero}
-\Delta u + Vu = (I + K(0))\langle x \rangle^{-s}g = 0,
\end{equation}
whence $u(x) = c_n \int_{\R^n} |x -y|^{-n+2} V(y) u(y) dy$, where $c_n|x|^{-n+2}$ is the fundamental solution of the Laplacian (for appropriate $c_n \in \R$ depending on $n$).

Any function in $H^2(\R^3)$ has a continuous, bounded representative. This is also true for members of the Sobolev space $H^3(\R^4)$. Thus, our earlier representation $u = \langle x \rangle^{s'}  R_{0, s', s}(0) g \in \langle x \rangle^{s'} H^2(\R^n)$ shows $\langle x \rangle^{-s'}u$ is bounded when $n = 3$. Our extra condition when $n = 4$, that the first distributional derivatives of $V$ belong to $L^\infty(\R^4)$, implies $\langle x \rangle^{-s'}u$ is bounded  in that case too. This follows by differentiating \eqref{use eqn lambda equals zero} in the sense of distributions
\begin{equation*}
\begin{split}
\Delta (\partial^\ell_{x_j} (\langle x \rangle^{-s'} u)) &= (\partial^\ell_{x_j}   \Delta(\langle x \rangle^{-s'} u) \\
& = \partial^\ell_{x_j} \big( ( \Delta \langle x \rangle^{-s'} ) u + 2 (\nabla \langle x \rangle^{-s'}) \cdot \nabla u + \langle x \rangle^{-s'} Vu \big), \quad 0 \le \ell \le 1,\, 1 \le j \le n.
\end{split}
\end{equation*}

 Because $u \in H^2_{\text{loc}}(\R^n)$ and $V \ge 0$, by Green's formula, for any $r > 0$, 
\begin{equation} \label{green's formula}
\| \nabla u \|^2_{L^2(B(0,r))}  \le \| \nabla u \|^2_{L^2(B(0,r))} + \langle Vu , u\rangle_{L^2(B(0,r))} = - \real \langle u, \partial_r u \rangle_{L^2(\partial B(0,r))}.
\end{equation}
We show there exist $C, \delta > 0$ so that
\begin{equation} \label{bd u partial r u}
\begin{gathered}
|u(x)| \le C \langle x \rangle^{-\frac{n-1}{2}}, \\
|\partial_r u(x)| \le C \langle x \rangle^{-\frac{n-1}{2} -\delta}.
\end{gathered}
\end{equation}
Since the $n - 1$ dimensional volume of $\partial B(0,r)$ is $O(r^{n-1})$, \eqref{bd u partial r u} implies the right side of \eqref{green's formula} tends to zero as $r \to 0$. Therefore $u$ must be a constant, and by \eqref{bd u partial r u} that constant must be zero.

We have 
\begin{equation*}
\begin{split}
|u(x)| &\le  \int_{|x - y| \ge 1, \, |x| > 2y} |x -y|^{-n + 2} |V(y)u(y)| dy\\
&+ \int_{|x - y| \ge 1, \, |x| \le 2y} |x -y|^{-n+2} |V(y)u(y)| dy \\
&+ \int_{|x - y| < 1} |x -y|^{-n + 2} |V(y)u(y)| dy \qefed I_1(u)(x) + I_2(u)(x) + I_3(u)(x).
\end{split}
\end{equation*}
For $|x| \gg 1$ and $|x -y| < 1$, $|y| \ge ||x| - |x-y|| \ge |x|/2$. So by \eqref{srp decay appendix} and $s- 1 \ll 1$,
\begin{equation*}
| \langle y \rangle^{s} V(y) | = O( \langle y \rangle^{s  - \rho}) =  O( \langle x \rangle^{-\frac{n-1}{2}}). 
\end{equation*}
Thus
\begin{equation*}
I_3 (u) = O(\langle x \rangle^{-\frac{n-1}{2}}) \int_{|x - y| < 1} |x -y|^{-n+ 2} |\langle y \rangle^{-s} u(y)| dy.
\end{equation*}
As discussed above $\langle \cdot \rangle^{-s} u$ is bounded, hence
\begin{equation} \label{use Linfty}
\int_{\{y \in \R^n : |x - y| < 1\}} |x -y|^{-n+2} |\langle y \rangle^{-s} u(y)| dy \le  \|\langle \cdot \rangle^{-s} u\|_{L^\infty(\R^n)} \int_{\{y \in \R^n : |y| < 1\}}  |y|^{-n+2} dz<\infty.  
\end{equation}
Next, if $|x - y| \ge 1$ and $|x| > 2|y|$, then $|x-y| \ge |x| - |y| \ge |x|/2$. By $s - 1 \ll 1$ and \eqref{srp decay appendix}, $\langle \cdot \rangle^{s} V = O(\langle \cdot \rangle^{s-\rho})$ belongs to $L^2(\R^n)$. Thus
\begin{equation*}
I_1(u) = O(\langle x \rangle^{-n + 2} ) \int_{|x - y| \ge 1} |Vu| dy = O(\langle x \rangle^{-\frac{n -1}{2}} )\| \langle \cdot \rangle^{s} V \|_{L^2} \| \langle \cdot \rangle^{-s} u \|_{L^2}. 
\end{equation*}
Finally, if $|x - y| \ge 1$ and $|x| \le 2|y|$, then $|V(y)| = O( \langle x \rangle^{-\frac{n-1}{2} - (s -1)} \langle y \rangle^{\frac{n-1}{2} + (s-1) - \rho})$. Since $s - 1 \ll 1$ and \eqref{srp decay appendix} imply $\langle \cdot \rangle^{2s + \frac{n-1}{2} -1 - \rho} \in L^2(\R^n)$ when $n =3$ or $4$, 
\begin{equation*}
\begin{split}
I_2(u) &= O(\langle x \rangle^{-\frac{n-1}{2}- (s-1)})  \int_{|x - y| \ge 1} |   \langle y \rangle^{1 - \rho}u| dy\\ &= O(\langle x \rangle^{-\frac{n-1}{2} - (s-1)})  \| \langle \cdot \rangle^{2s + \frac{n-1}{2} - 1- \rho} \|_{L^2} \| \langle \cdot \rangle^{-s} u \|_{L^2}. 
\end{split}
\end{equation*}

To get the bound on $\partial_r u$ in \eqref{bd u partial r u}, we proceed in a similar manner but with some minor modifications. Since, 
\begin{equation*}
\partial_{x_j} |x - y|^{-n + 2} = (-n + 2) |x -y|^{-n + 1} \frac{x_j - y_j}{|x-y|},
\end{equation*}
we have
\begin{equation*}
\begin{split}
|\partial_ru(x)| &\le  \int_{|x - y| \ge 1, \, |x| > 2y} |x -y|^{-n+1} |V(y)u(y)| dy\\
&+ \int_{|x - y| \ge 1, \, |x| \le 2y} |x -y|^{-n + 1} |V(y)u(y)| dy \\
&+ \int_{|x - y| < 1} |x -y|^{-n +1} |V(y)u(y)| dy \qefed I_1(\partial_ru)(x) + I_2(\partial_ru)(x)  + I_3(\partial_ru)(x).
\end{split}
\end{equation*}
Since $\int_{|y| < 1} |y|^{-n+1}dy < \infty$, to bound $I_3(\partial_ru)$ we again give an estimate alongs the lines \eqref{use Linfty}, finding $I_3(\partial_r u) = O(\langle x \rangle^{s - \rho}) = O(\langle x \rangle^{-\frac{n -1}{2} - \delta})$ for some $\delta > 0$. The estimates of $I_1(\partial_r u)$ and $I_2(\partial_ru)$ also follow those of $I_1(u)$ and $I_2(u)$, respectively, and yield $I_1(\partial_r u) = O(\langle x\rangle ^{-n +1})$ and $I_2(\partial_r u) = O(\langle x \rangle^{-\frac{n-1}{2} - (s-1)})$, both of which are $O(\langle x \rangle^{-\frac{n -1}{2} - \delta})$ for some $\delta > 0$.

\section{Useful estimates} \label{inequalities appendix}

\begin{lemma} \label{hardy lemma}
Let $n \ge 3$. Then,
\begin{equation} \label{hardy n ge 3}
\| r^{-1} u \|^2_{L^2} \le \Big( \frac{2}{n -2} \Big)^2 \| \nabla u \|^2_{L^2}, \qquad u \in H^1(\R^n). 
\end{equation} 
In dimension two,
\begin{equation} \label{hardy n ge 2}
\|r^{-1/2} u \|^2_{L^2} \le \| u\|^2_{L^2} + \| \nabla u \|^2_{L^2}, \qquad u \in H^1(\R^2). 
\end{equation}
\end{lemma} 

\begin{proof}
Both inequalities are standard. The estimate for $n \ge 3$ appears in the proof of \cite[Proposition 6]{fa67}. We are not aware of an accessible reference for the dimension two case, so we include a short proof here for completeness.

Since $C^\infty_0(\R^2)$ is dense in $H^1(\R^2)$, it suffices to prove \eqref{hardy n ge 2} for $u \in C^\infty_0(\R^2)$. Using polar coordinates,
\begin{equation} \label{dim two decompose by polar}
\int_{\R^2} r^{-1} |u|^2 dx = \int_{\US^1} \int^\infty_0 |u(r, \theta)|^2 dr d\theta.
\end{equation}
Integrating by parts
\begin{equation*}
\begin{split}
 \int^\infty_0 |u(r, \theta)|^2 dr &=  \int^\infty_0 |u(r, \theta)|^2 r' dr \\
 &= -2 \real \int^\infty_0 u(r, \theta) \overline{u}'(r, \theta) r dr \\
 &\le \int_0^\infty |u|^2 rdr + \int_0^\infty |u'|^2 r dr \\
 &\le  \int_0^\infty |u|^2 rdr + \int_0^\infty |\nabla u|^2 r dr.
 \end{split}
\end{equation*}
We conclude the proof of \eqref{hardy n ge 2} by integrating the last inequality over $\US^{1}$ and taking into account \eqref{dim two decompose by polar} \\
\end{proof}

\begin{lemma}
Let $m \ge 0$ and $\kappa > 0$. Then for any $0 < \nu < 1$,
\begin{equation*}
\int_0^\kappa \lambda^m \sin(t \lambda) d\lambda = O(t^{-\nu}), \qquad \text{as } t \to \infty.
\end{equation*}
\end{lemma}
\begin{proof}
Let $\nu < \nu_1 < 1$ such that $\nu_1(m+2) > 1 + \nu$. We split the integral at $\lambda = t^{-\nu_1}$:
\begin{equation*}
\int_0^\kappa \lambda^m \sin(t \lambda) d\lambda = \int_0^{t^{-\nu_1}} \lambda^m \sin(t \lambda) d\lambda + \int_{t^{-\nu_1}}^\kappa \lambda^m \sin(t \lambda) d\lambda \qefed I_1 + I_2.
\end{equation*}
For the first integral, we use the bound $|\sin(t\lambda)| \le t\lambda$:
\begin{equation*}
|I_1| \le t \int_0^{t^{-\nu_1}} \lambda^{m+1} d\lambda = \frac{t}{m+2}(t^{-\nu_1})^{m+2} = O(t^{1-\nu_1(m+2)}).
\end{equation*}
By our choice of $\nu_1$, this is $O(t^{-\nu})$. For the second integral, we integrate by parts:
\begin{equation*}
I_2 = \left[ -\frac{\cos(t\lambda)}{t} \lambda^m \right]_{t^{-\nu_1}}^\kappa + \frac{m}{t} \int_{t^{-\nu_1}}^\kappa \lambda^{m-1} \cos(t \lambda)  d\lambda.
\end{equation*}
The boundary terms are $O(t^{-1})$. The final term is also bounded by $O(t^{-\nu})$, since \\$\int_{t^{-\nu_1}}^\kappa \lambda^{m-1} \cos(t \lambda)d\lambda = O(\log t)$.\\
\end{proof}

\begin{lemma} \label{simple Sobolev est lem}
Suppose $T : L^2(\R^n) \to H^2(\R^n)$ is a bounded operator. For any $s > 0$, there exists $C > 0$ so that 
\begin{equation} \label{recast L2 to H2 bd}
\| \langle x \rangle^{-s} T \|_{L^2 \to H^2} \le C(\| \langle x \rangle^{-s} T \|_{L^2 \to L^2} + \| \langle x \rangle^{-s} \Delta T \|_{L^2 \to L^2}).
\end{equation}
\end{lemma}

\begin{proof}
Let $f \in L^2(\R^n)$ and put $u = Tf$. By the first line of \eqref{std elliptic thry}, there exists $C > 0$, whose precise value may change from line to line, so that 
\begin{equation} \label{apply std elliptic thry}
\| \langle x \rangle^{-s} u \|_{H^2} \le C ( \| \langle x \rangle^{-s} u \|_{L^2} + \| \Delta \langle x \rangle^{-s} u \|_{L^2})
\end{equation}
Then use the second of \eqref{std elliptic thry},
\begin{equation*}
\begin{split}
\| \Delta \langle x \rangle^{-s} u \|_{L^2} &\le \|[\Delta, \langle x \rangle^{-s}] u \|_{L^2} + \|\langle x \rangle^{-s} \Delta  u \|_{L^2} \\
&\le C\|\langle x \rangle^{-s} u \|_{H^1} + \|\langle x \rangle^{-s} \Delta  u \|_{L^2} \\
&\le C (\gamma^{-1} \| \langle x \rangle^{-s} u \|_{L^2} +  \gamma \| \Delta \langle x \rangle^{-s} u \|_{L^2}) +\|\langle x \rangle^{-s} \Delta  u \|_{L^2}, \qquad \gamma > 0.
\end{split}
\end{equation*}
Fixing $\gamma$ small enough yields,
\begin{equation*}
\| \Delta \langle x \rangle^{-s} u \|_{L^2} \le C(\| \langle x \rangle^{-s} u \|_{L^2} +\|\langle x \rangle^{-s} \Delta  u \|_{L^2}),
\end{equation*}
which in combination with \eqref{apply std elliptic thry} implies \eqref{recast L2 to H2 bd}.\\
\end{proof}

\section*{Declarations}

\noindent\textbf{Data Availability.} No datasets were generated or analyzed during the current study.

\medskip

\noindent\textbf{Funding.} J.S. acknowledges support from National Science Foundation (NSF) Division of Mathematical Sciences (DMS) award 2204322, a 2023 Fulbright Future Scholarship funded by the Kinghorn Foundation and hosted by the University of Melbourne, and a University of Dayton Research Council Seed Grant. M.Y. acknowledges support from an AMS-Simons Travel Grant and NSF award DMS-2509989. This article is based upon work supported by NSF DMS-1929284 while J.S. and M.Y. were in residence at the Institute for Computational and Experimental Research in Mathematics (ICERM) in Providence, RI, during the August 2024 workshop entitled \textit{Spectral Analysis for Schrödinger Operators}.

\end{document}